\documentclass[11pt,oneside,reqno]{amsart}
\usepackage[margin=1in]{geometry}
\usepackage{booktabs}
\usepackage[toc]{appendix}
\usepackage{booktabs}
\usepackage{tabularx}
\usepackage{booktabs}
\usepackage{ltablex}
\keepXColumns
\renewcommand{\arraystretch}{1.15}
\usepackage{threeparttable}
\usepackage{array} 

\usepackage{lipsum}

\newcommand\blfootnote[1]{%
 \begingroup
 \renewcommand\thefootnote{}\footnote{#1}%
 \addtocounter{footnote}{-1}%
 \endgroup
}

 \usepackage{tikz}
\usetikzlibrary{arrows.meta,positioning,calc}
\usepackage{graphicx} % for \resizebox

\usepackage{latexsym}
\usepackage[maxbibnames=99,citestyle=numeric-comp]{biblatex}
\usepackage{amsmath, amsthm, amssymb, latexsym,epsfig,amsthm, enumerate,multicol,wasysym}
\usepackage{bm}
\usepackage{mathtools}
\usepackage{color}
\usepackage{graphicx}
\usepackage{nccrules}
\usepackage{xfrac}
\usepackage{hyperref}
 \hypersetup{
 colorlinks,
 citecolor=red,
 filecolor=red,
 linkcolor=red,
 urlcolor=black
}
\usepackage{textcomp}
\usepackage{tikz}
\usepackage[dvipsnames]{xcolor}
\usepackage{comment, todonotes}
\usepackage{enumitem}

\usepackage{emptypage}
\usepackage{placeins}

\numberwithin{equation}{section}
\theoremstyle{plain}
\newtheorem{theorem}{Theorem}[section]
\newtheorem{lemma}[theorem]{Lemma}

\newtheorem{proposition}[theorem]{Proposition}

\theoremstyle{definition}
\newtheorem{definition}[theorem]{Definition}

\newtheorem{case[theorem]}{Case}

\usepackage{caption}
\usepackage{mwe}
\usepackage{wrapfig}

\theoremstyle{remark}
\newtheorem{remark}[theorem]{Remark}

\numberwithin{equation}{section}

\newcommand{\abs}[1]{\lvert#1\rvert}
\def\R{\mathbb{R}}

\usepackage[normalem]{ulem} 
\makeatletter
\renewcommand\@makefntext[1]{%
 \noindent
 \makebox[1em][r]{\@makefnmark}#1}
\makeatother 
\usepackage{ulem} 
\usepackage{float}

\newcommand{\Z}{\mathbb Z}

\newcommand{\cB}{\mathcal B}

\newcommand{\cH}{\mathcal H}

\newcommand{\cM}{\mathcal M}
\newcommand{\cI}{\mathcal I}

\newcommand{\supp}{\mathrm{supp}}
\newcommand{\dist}{\mathrm{dist}}

\newcommand{\Iint}{\mathcal I^{\mathrm{int}}}
\newcommand{\Ibdry}{\mathcal I^{\mathrm{bdry}}}

\def\R{\mathbb{R}}

\title{Wave Packets and Eigenvalue Estimates for Limiting Operators on the Disk}
\author{Kevin Hughes, Arie Israel, Azita Mayeli}
\date{\today}

\begin{document}
\maketitle

\begin{abstract}
We study two-dimensional spatio--spectral limiting operators
\[
T_R := P_{D(R)} B_S P_{D(R)} : L^2(\mathbb{R}^2) \rightarrow L^2(\mathbb{R}^2),
\]
where $D(R)$ is a disk of radius $R>1$,
$S\subset\mathbb{R}^2$ is a domain with well--shaped boundary, $P_{D(R)}$ is the orthogonal projection on the subspace of functions supported on $D(R)$, and $B_S$ is the orthogonal projection on the subspace of functions whose Fourier transform is supported on $S$. 
We construct a disk-adapted wave-packet frame for $L^2(D(R))$ with frame bounds uniform in $R$ using Gevrey-$s$ cutoffs ($s>1$) to obtain near-exponential Fourier localization.
Exploiting these localization estimates, we bound the size of the eigenvalue plunge--region for $T_R$ and prove that for each $s>1$ and each $\varepsilon\in(0,1/2)$,
\[
\#\{k : \lambda_k(T_R)\in(\varepsilon,1-\varepsilon)\}
=
O\!\left(R (\log(R/\varepsilon))^{1+2s}\right),
\]
with constants depending on $s$ and the geometric parameters of $S$.
This bound improves existing plunge--region estimates in the classical setting where both domains are disks, when $\varepsilon$ scales like $R^{-\nu}$ for a fixed $\nu > 0$. 
By an affine transformation, the same result holds if $D(R)$ is a scaled ellipse.

\end{abstract}

\blfootnote{A.~ Israel was supported by the Air Force Office of Scientific Research, under award FA9550-19-1-0005 and the National Science Foundation grant DMS-2453770. A.~Mayeli was supported in part by the National Science Foundation grant DMS-2453769, the AMS-Simons Research Enhancement Grant, and the PSC-CUNY grants 67807-00 56 and 67807-00 57.}

% \subjclass[2020]{Primary 35S30; Secondary 35P20, 42C15}
%\thispagestyle{empty}
\setcounter{tocdepth}{1}
\tableofcontents

\section{Introduction}
\subsection{Background and SSLOs}
We quantify the tradeoff between spatial and frequency localization via eigenvalue estimates for spatio--spectral limiting operators (SSLOs). 
Given measurable sets \( F,S\subset\mathbb{R}^d\), define the orthogonal spatial projection \(P_F : L^2(\R^d) \rightarrow L^2(\R^d)\), \(P_F f := \mathbf{1}_F f\) and the band--limiting projection \(B_S : L^2(\R^d) \rightarrow L^2(\R^d)\), \(B_S f := \mathcal{F}^{-1}(\mathbf{1}_S \mathcal{F} f)\) to form the spatio--spectral limiting operator \(T_{F,S}= P_F B_S P_F\) (or equivalently \(B_S P_F B_S\), which has the same non-zero spectrum).  Here, \(\mathcal{F}\) denotes the Fourier transform on \(L^2(\mathbb{R}^d)\) and \(\mathcal{F}^{-1}\) denotes its inverse; both are normalized as in \eqref{def:FT}. 
If \(F\) and \(S\) have finite Lebesgue measure, then \(T_{F,S}\) is compact and self--adjoint on $L^2(\R^d)$, with eigenvalues $1 \geq \lambda_1(T_{F,S}) \geq \lambda_2(T_{F,S}) \geq \cdots > 0$. The uncertainty principle asserts that all of the eigenvalues are in fact strictly less than \(1\). 

For any subset \(F\) in \(\mathbb{R}^d\) and \(R>0\), define the isotropic dilation 
\begin{equation}\label{def:dilates}
F(R):= \{Rx : x \in F\}.
\end{equation}
We have the trace identity \(\mathrm{tr}(P_{F(R)} B_{S} P_{F(R)}) = (2\pi)^{-d} |F|\cdot|S| \cdot R^d\) provided $F$ and $S$ have finite Lebesgue measure. Ordering the eigenvalues in non-increasing order, one expects roughly \((2\pi)^{-d} |F|\cdot|S| \cdot R^d \) eigenvalues close to \(1\), with the remainder rapidly decaying to \(0\). Heuristically, the eigenvalues of such operators are clustered near \(1\) or \(0\). 
To make this precise, let \(N_\varepsilon(R)\) denote the number of eigenvalues of \(P_{F(R)} B_{S} P_{F(R)}\) that are greater than \(\varepsilon \in (0,1)\). A classical result of Landau \cite{Landau75} shows that
\begin{equation}\label{LandauWeyl}
\lim_{R \to \infty} R^{-d}N_\varepsilon(R) = (2\pi)^{-d}|F|\cdot|S|
.\end{equation}
Thus, the number of eigenvalues of \(P_{F(R)} B_{S} P_{F(R)}\) exceeding a fixed threshold obeys a Weyl-type law as \(R\to\infty\). The original argument in \cite{Landau75} does not provide quantitative rates, and a natural question is to understand the distribution of the eigenvalues of spatio-spectral limiting operators for a fixed dilation parameter $R$. 

Recent works, such as \cite{israel15eigenvalue,israelmayeli2023acha,marceca2023,hughes2024eigenvalue}, focused on making Landau's asymptotic results quantitative by proving explicit bounds on the size of the \emph{plunge region} \(\{\lambda_k(T)\in(\varepsilon,1-\varepsilon)\}\), where $T$ is an SSLO, \(\lambda_k(T)\) are its eigenvalues, and \(\varepsilon \in (0,1/2)\). 
Bounding the size of the plunge region allows us to quantify the clustering behavior of eigenvalues near $1$ and $0$. 
Asymptotic estimates on the size of the plunge region were developed first in the one--dimensional setting, starting with the work of Landau and Landau--Widom in \cite{LandauWidom80,Landau93}. 
Recently, explicit non--asymptotic estimates in the same setting were obtained in \cite{israel15eigenvalue,Bonami21,Kulikov24,davenport2021improved,osipov2013}. 
Related discrete analogues, for time--frequency operators, appear in \cite{Bell5,davenport2017,davenport2019fast,davenport2021improved}, and for higher--dimensional SSLOs, appear in \cite{marceca2023, Li2025Approximating, GJMP_prolateHD}. 
% Higher--dimensional discrete SSLOs also in \cite{marceca2023, Li2025Approximating, }
Recent works \cite{israelmayeli2023acha,marceca2023,hughes2024eigenvalue} developed estimates for the number of eigenvalues in the plunge region in settings where \(F\) and \(S\) are bounded domains in \(\mathbb{R}^d\) (with \(d \geq 2\)) satisfying various geometric regularity assumptions on the boundaries of \(F\) and \(S\).

\subsection{Main results}

The present paper provides a strong estimate for the number of eigenvalues in the plunge region when the domains \(F\) and \(S\) are taken from a restricted geometric class. 

\begin{theorem}\label{thm:newmain} 
Assume \(d=2\). 
Let \( F \subset \R^2 \) be an ellipse centered at the origin, and let \(S \subset \R^2\) be a well--shaped domain (see Definition~\ref{def:well--shaped}). 
For each \(R>1\), define the spatio--spectral limiting operator 
\begin{equation}\label{def:SSLO}
T_R := P_{F(R)} B_{S} P_{F(R)} : L^2(\mathbb{R}^2)\to L^2(\mathbb{R}^2)
\end{equation}
whose eigenvalue sequence \(\{\lambda_k(T_R)\}_{k\geq 0}\) lies in \([0,1]\). 
For each $s > 1$, there exists a constant \(C_{s,S,F}\) depending only on \(s,S,\) and \(F\) such that for all \(\varepsilon \in (0,1/2)\) and \(R > 1\), 
\begin{equation}\label{bd:plunge}
\#\{\lambda_k(T_R)\in(\varepsilon,1-\varepsilon)\}
\leq C_{s,S, F} R \big(\log(R/\varepsilon)\big)^{1+2s}.
\end{equation}
\end{theorem}

See Definition~\ref{def:well--shaped} below for the definition of a ``well--shaped'' domain, which includes convex domains among other domains with sufficiently regular boundaries. 
We bring to attention that Theorem~\ref{thm:newmain} gives new bounds for the size of the plunge region of limiting operators when \(S\) and \(F\) are chosen to be the unit disk. See Section \ref{comparison} for a comparison with previous results. 

% \begin{remark}\label{reduction}
By invariance of the Fourier transform under invertible affine maps and the affine invariance of the class of well--shaped domains, Theorem~\ref{thm:newmain} reduces to the case in which \(F\) is the unit disk $D$, \(R=2^j\) for some \(j\in\mathbb{N}\), and \(S\) has diameter at most $1$.
% \end{remark}
% 
Henceforth, we assume that \(F(R) = D(R)\) is the two-dimensional disk of radius $R$ centered at the origin and \(R\) is a positive integer power of \(2\). 
We prove Theorem~\ref{thm:newmain} by constructing a system of wave packets on \(D(R)\) adapted to \(S\), satisfying the energy estimate (II.5) hypothesized in \cite{HughesIsraelMayeli_SampTA2025}. Specifically, we prove:

\begin{theorem}\label{thm2}
Fix a dyadic integer $R \geq 2$ and $s > 1$. There exists a frame \(\{\psi_\nu\}_{\nu \in \cI}\) for $L^2(D(R))$, with absolute frame constants $0 <A < B < \infty$ (independent of \(R\) and \(s\)), such that the following holds. Let  \(S \subset \R^2\) be a well--shaped domain. For each \(\varepsilon \in (0,1/2)\) there exists a partition of the index set \(\cI=\cI_1\cup \cI_2\cup \cI_3\) and a constant \(C>0\) determined by $s$ and the geometric parameters of $S$ (and independent of \(R\) and \(\varepsilon\)) satisfying
\begin{equation}\label{eqn:energy1}
\sum_{\cI_1}\|\widehat{\psi_\nu}\|_{L^2(S)}^2 + 
%+
\sum_{\cI_2}\|\widehat{\psi_\nu}\|_{L^2(\R^2\setminus S)}^2
\leq \varepsilon^2
\end{equation}
and 
\begin{equation}\label{eqn:residual_bd}
\#\cI_3 \le C R \log(R/ \varepsilon)^{1+2s}.
\end{equation}
\end{theorem}

For \(\varepsilon \in (0,1/2)\), we refer to estimates \eqref{eqn:energy1} and \eqref{eqn:residual_bd} as \emph{energy concentration estimates of order \( \varepsilon\) on \(D(R)\times S\)}. We also refer to the set \(\cI_3\) as the \emph{residual set}. 

In the course of the paper, we break this theorem into Propositions \ref{prop2} and \ref{prop1} whose combination immediately yields the theorem.

This theorem addresses Open Problem~2 from \cite{HughesIsraelMayeli_SampTA2025}, which asks to construct a wave packet frame on the $d$-dimensional unit ball $B_d(0,1)$ that admits a decomposition satisfying energy estimates as in \eqref{eqn:energy1}, with the cardinality of the residual set bounded by
\[
\#\cI_3 \le C R^{d-1}\log(R/ \varepsilon)^{J},
\]
for some $J>0$ and  \(C>0\). Theorem~\ref{thm2} answers this problem for $d=2$. The problem remains open in dimensions $d\geq 3$.

\subsection{Comparison of \eqref{bd:plunge} to previous estimates}\label{comparison}
Let \(F,S\) be subsets of \(\R^d\) with finite Lebesgue measure, \( T := P_{F} B_{S} P_{F} \) be the associated SSLO and \(\{\lambda_k(T)\}\) be the eigenvalues of the SSLO \(T\). 
For \(\varepsilon \in (0,1/2)\), we want sharp bounds for the cardinality of the plunge region which we denote by 
\[
M_\varepsilon(T) :=\# \{\lambda_k(T)\in(\varepsilon,1-\varepsilon)\}.
\]

For the operator $T_R=P_{F(R)}B_S P_{F(R)}$ with $S$ well--shaped, we compare the plunge--region bound \eqref{bd:plunge} to previous estimates for \(M_\varepsilon(T_R)\). 
To focus the discussion and to simplify our comparison, we only state the previous bounds in dimension $d=2$. 
For the remainder of this section, ``domains" will refer to bounded, open, connected subsets of the Euclidean plane \(\mathbb{R}^2\).

For the following comparisons, assume that \(\varepsilon \in (0,1/2)\) and \(R>1\). When \(F\) is chosen to be the unit square \([0,1]^2\) so that $F(R) = [0,R]^2$, and \(S\) is chosen to be a coordinate--wise symmetric convex body of diameter at most $1$, \cite{israelmayeli2023acha} proved that 
\begin{equation}\label{bd:IM24}
M_\varepsilon(T_R)
\leq C \max\{ R (\log(R/\varepsilon) \big)^{5/2}, (\log(R/\varepsilon) \big)^{5} \}
\end{equation}
where the implicit constant does not depend on \(R,\varepsilon\) nor on \(S\). For fixed $\varepsilon$ and all sufficiently large $R$, the first term on the right-hand side of \eqref{bd:IM24} dominates the second term and is sharper in the log-power than Theorem~\ref{thm:newmain}.

To our knowledge, the only other works obtaining quantitative eigenvalue bounds for limiting operators in higher dimensions are \cite{marceca2023} and \cite{hughes2024eigenvalue}. Suppose $F$ and $S$ are domains whose boundaries  \(\partial F\) and \(\partial S\) are maximally Ahlfors regular with finite regularity constants. Then \cite{marceca2023} proved that for every $\alpha\in (0,1/2]$,
\begin{equation}\label{bd:MRS24}
M_\varepsilon(T_R)
\leq C_{\alpha,F, S} R \big( \log(R/\varepsilon) \big)^{4(1+\alpha)+1}. 
\end{equation} 
Subsequently, for the same class of domains, we proved in \cite{hughes2024eigenvalue} that 
\begin{equation}\label{bd:HIM25}
M_\varepsilon(T_R)
\leq C_{F,S} R \Big( (\log R) (\log\varepsilon^{-1})^2 + ( \log R)^3 (\log \varepsilon^{-1}) \Big)
.\end{equation}
A direct comparison of the right–hand side of \eqref{bd:HIM25} with those of \eqref{bd:MRS24} and \eqref{bd:plunge} follows from the elementary inequality
\[
(\log R) (\log\varepsilon^{-1})^2 + ( \log R )^3 (\log \varepsilon^{-1}) 
\leq 
( \log(R/\varepsilon))^4,
\]
which is a consequence of the binomial expansion of $(a+b)^4$ with $a=\log R$ and $b = \log \varepsilon^{-1}$. This estimate is essentially sharp when $R$ grows polynomially in $\varepsilon^{-1}$. In this regime, \eqref{bd:HIM25} improves \eqref{bd:MRS24} by saving one power of $\log(R/\varepsilon)$. In the same regime, Theorem~\ref{thm:newmain} saves almost one additional power of $\log(R/\varepsilon)$ over \eqref{bd:HIM25}. 
Alternatively, for fixed \(\varepsilon\), the bound \eqref{bd:HIM25} is slightly better than \eqref{bd:plunge} of Theorem~\ref{thm:newmain} in terms of the dependence on $R$.

\subsection{Heuristic} 
All known estimates of the plunge region are governed by boundary geometry of the underlying domains. From these estimates, the heuristic seems to be that \emph{\(M_\varepsilon(T_{F,S})\) is controlled by the codimension 1 Hausdorff measure of the boundaries of \(F\) and \(S\)} up to polylogarithmic factors of \(\varepsilon\) and these measures. 
In dimensions \(d \geq 2\), when one of the domains is isotropically dilated by a factor of \(R>1\), this corresponds to an error term of size \(R^{d-1}\) up to additional polylogarithmic factors. 
The estimate \eqref{bd:plunge} of our main theorem is consistent with this heuristic when $d=2$, since it shows that \(M_\varepsilon(T_R)\) grows at most linearly in \(R\), up to polylogarithmic losses.

% This heuristic is also supported by work on Wiener--Hopf operators where a logarithmic term is necessary. See \cite{widom1982class,sobolev2013widombook,sobolev2014widom} for more on Widom's conjecture and the presence of logarithmic terms in two--term asymptotics for Weyl--type laws. When comparing note that \(\varepsilon\) does not appear in the terms because of the use of gauge functions in the trace which are smooth and 0 at 0.

% \begin{problem}
% Are there asymptotics for the number of eigenvalues in the plunge region? If so, what is the precise relationship with the geometry of the underlying domains? 
% \end{problem}

\subsection{Method and organization} 

As stated in the discussion following Theorem~\ref{thm:newmain}, without loss of generality, we may assume that the spatial domain $F(R)$ is the disk $D(R)$ of radius $R$. 
Our approach parallels the strategy used in the hypercube setting by the second and third authors of this paper in \cite{israelmayeli2023acha}. 
The key steps in \cite{israelmayeli2023acha} were 
\begin{itemize}
\item[Step~1:] 
Construct tensor--product wave packets adapted to the geometry of the hypercube to produce an orthonormal basis satisfying energy concentration estimates analogous to \eqref{eqn:energy1} and \eqref{eqn:residual_bd}. See Proposition~5.3 therein. 
\item[Step~2:] 
Use a frame--based eigenvalue counting lemma, converting energy concentration bounds for an orthonormal basis into estimates for the number of eigenvalues in the plunge--region. See Lemma~4.1 therein.
\end{itemize}

Theorem~\ref{thm2} plays the role of Step~1 and addresses the crucial analytic difficulty of constructing Fourier localized wave packets compatible with the geometry of the disk. This difficulty arises because the curved boundary $\partial D(R)$ introduces tangential localization effects that require boundary--adapted wave packets in polar coordinates. 
To construct the wave packet family in Theorem~\ref{thm2}, we introduce a Whitney--type radial--angular sectorization of \(D(R)\), employ smooth cutoffs, and define two subfamilies: 
\begin{itemize}
    \item 
    Interior packets with linear phase supported away from the boundary of $D(R)$, and 
    \item 
    Boundary packets whose phase is linear in polar coordinates (and hence nonlinear in Cartesian coordinates) supported near the boundary of $D(R)$. 
\end{itemize}
The smooth cutoffs are chosen to lie in a Gevrey class. This allows us to obtain nearly exponential decay of the Fourier transforms of the wave packets. See Lemmas~\ref{lem:typeI} and \ref{lem:type2_decay}. In turn, this Fourier decay is used to prove the desired energy concentration estimates. 

In contrast to the construction in \cite{israelmayeli2023acha}, the wave packets constructed in this paper do not form an orthonormal basis but instead form a unit norm frame for \(L^2(D(R))\) satisfying the bounds in Theorem~\ref{thm2}. 
This suffices for our purposes. Lemma~\ref{lem:Romero_lem} plays the role of Step~2, and applying it, Theorem~\ref{thm2} implies Theorem~\ref{thm:newmain}.

The paper is organized as follows. 
In \S\ref{sec:prelim} we fix notation, define and discuss well--shaped sets, recall basic properties of the spatio--spectral limiting operator $T_R=P_{D(R)}B_{S}P_{D(R)}$, and record the frame--theoretic tools used throughout. 

In \S\ref{wave-packet-section} we construct a wave packet system adapted to the disk $D(R)$. In particular, we introduce a Whitney--type radial--angular sectorization of the disk (\S\ref{sec:sector}), build Gevrey cutoffs adapted to this decomposition (\S\ref{sec:cutoffs}), and define interior packets versus boundary packets (\S\ref{defn:WP}). 
We conclude this section by proving Proposition~\ref{prop2}, which says that our system of wave packets forms a unit--norm frame for $L^2(D(R))$ with absolute frame bounds (\S\ref{subsec:frame_property}).
 
In \S\ref{sec:Fourier localization and energy concentration estimates} we establish the quantitative Fourier localization estimates that drive the spectral analysis.
We treat interior packets in \S\ref{sec:energy_type1} and boundary packets in \S\ref{sec:energy_type2}, leading to the energy concentration stated in Proposition~\ref{prop1}.

In \S\ref{sec:mainresult} we combine Proposition~\ref{prop1} with an eigenvalue counting lemma for positive contractions acting on a Hilbert space equipped with a unit--norm frame (Lemma~\ref{lem:Romero_lem}) to deduce the plunge--region bound in Theorem~\ref{thm:newmain}.

The appendices collect background on Gevrey classes and the decay of Fourier transforms of Gevrey functions (Appendix~\ref{appendix:Gevrey}), an explicit construction of the radial and angular cutoff families (Appendix~\ref{cutoffs}), and a table of notations used in the text (Appendix~\ref{sec:notations}).

\subsection{Acknowledgements}

We thank Martina Neuman for providing valuable comments on our paper.

% \section{Preliminaries and tools}\label{sec:prelim}
\section{Definitions and notation}\label{sec:prelim}
Throughout, $X \lesssim Y$ means $X \leq CY$ for an absolute constant $C>0$ independent of $X,Y$, the parameter $R$ and of all auxiliary indices unless explicitly indicated. We write $X \approx Y$ when $X \lesssim Y$ and $Y \lesssim X$. By $X \lesssim_s Y$ and $X \approx_s Y$, we allow the implicit constants to depend on a parameter $s$.
For a measurable set $F \subset \mathbb{R}^2$, we denote by $\mathbf{1}_{F}$ its indicator function and by $|F|$ its Lebesgue measure.
A domain or body means a bounded, open, connected subset of the Euclidean plane \(\mathbb{R}^2\). 
Let \(\partial S\) denote the topological boundary of a set \(S \subset \mathbb{R}^2\). 
Let $\cH^1(\Gamma)$ denote the one--dimensional Hausdorff measure of a set $\Gamma \subset \R^2$. We write $\log x$ for the base-$2$ logarithm.

Theorem~\ref{thm:newmain} is stated and proved for well--shaped frequency domains which we define now. 
\begin{definition}\label{def:well--shaped}
A domain $S \subset \R^2$ is \emph{well--shaped} if its boundary $\partial S$ has finite $\cH^1$ measure, and if there exists a constant \(C>0\)  such that for all $t>0$, we have the Minkowski-type upper bound on the volume of the $t$-neighborhood of $\partial S$,
\begin{equation}\label{bd:well--shaped}
| \{ x \in \R^2 : \dist(x,\partial S) \leq t \} | \leq C \max\{ t^2, t \cH^1( \partial S)\}
.\end{equation} 
Throughout, \(\dist\) denotes the usual Euclidean distance. 
\end{definition}
\begin{remark}\label{rem:well-shaped-dilationinvariance}
The class of well–shaped domains is invariant under isotropic dilations with the same constant $C$. Indeed, if $a>0$, then $\partial(aS) = a\partial S$ and
$\cH^1(\partial(aS)) = a \cH^1(\partial S)$. By homogeneity of the distance function 
\[
\dist(x,\partial(aS))\le t \iff \dist(x/a,\partial S)\le t/a,
\]
so $\{x:\dist(x, \partial(aS))\le t\}=a\{y:\dist(y,\partial S)\le t/a\}$ and hence its area scales by $a^2$.
Applying \eqref{bd:well--shaped} to $S$ with $t/a$ gives
\[
|\{x:\dist(x,\partial(aS))\le t\}|
\le C\max\{t^2, t \cH^1(\partial(aS))\},
\]
and $aS$ is well--shaped with the same constant $C$.
\end{remark}

\begin{remark}
Variants of the well--shaped condition appear in many areas of mathematics. See \cite{Schmidt, Shparlinski, Kerr} for its use in number theory. 
Many important classes of domains satisfy the well--shaped condition; these include convex domains (\cite{Schmidt}), domains with maximally Ahlfors regular boundary (\cite{marceca2023} \S 
2), and domains with finite Lipschitz boundary (\cite{lang2013algebraic} Chapter~VI, \S 2). 
\begin{comment}
We also record the following result: A domain $S\subset\R^2$ with connected boundary of finite $\cH^1$ measure is well--shaped. To see this, we note if $\Gamma = \partial S$ is connected and $x \in \Gamma$ then
\[
\cH^1(\Gamma \cap B(x,r)) \geq r \mbox{ for all } 0 < r \leq \cH^1(\Gamma).
\]
Therefore, the boundary $\Gamma$ is maximally Ahlfors regular (see \cite{marceca2023} for a discussion of maximal Ahlfors regularity). We deduce from $\S 2$ of \cite{marceca2023} that $S$ is well--shaped with a uniform constant $C$.
\end{comment}
\end{remark}

\begin{lemma}[Lattice point counts for well-shaped domains] \label{lem:lattice_near_dilated_bd}
Let $S\subset\R^2$ be well--shaped with constant $C>0$ in \eqref{bd:well--shaped}. If $t \geq \sqrt{2}/2$, then 
\begin{equation}\label{eqn:lattice_near_dilated_bd}
\#\{\mathbf m\in\Z^2:\dist(\mathbf m,\partial S)\le t\}
\le 4C\max\{t^2, t \cH^1(\partial S)\}.
\end{equation}
\end{lemma}

\begin{proof}
Set 
\[
E:=\{\mathbf m\in\Z^2:\dist(\mathbf m,\partial S)\le t\}.
\]
Since $[-1/2,1/2]^2\subset B(0,\sqrt2/2)$, we have the inclusion 
\[
E+[-1/2,1/2]^2 \subset \{x\in\R^2:\dist(x,\partial S)\le t+\tfrac{\sqrt2}{2}\}.
\]
Taking Lebesgue measure and using $|E+[-1/2,1/2]^2|=\#E$ gives the inequality
\[
\#E \le \Big|\{x\in\R^2:\dist(x,\partial S)\le t+\tfrac{\sqrt2}{2}\}\Big|.
\]
Applying \eqref{bd:well--shaped} with constant \(C\), we obtain 
\[
% \#\{\mathbf m\in\Z^2:\dist(\mathbf m,\partialS)\le t\}
\#E 
\le C\max\Big\{(t+\tfrac{\sqrt2}{2})^2, (t+\tfrac{\sqrt2}{2}) \cH^1(\partial S) \Big\}.
\]
If $t\ge \sqrt2/2$, then $t+\sqrt2/2\le 2t$ and \eqref{eqn:lattice_near_dilated_bd} follows.
\end{proof}

For $f \in L^1(\mathbb{R}^2)\cap L^2(\mathbb{R}^2)$ we use the Fourier transform
\begin{equation}\label{def:FT}
\widehat{f}(\xi) = \frac{1}{2\pi} \int_{\mathbb{R}^2} f(x)e^{- i x\cdot \xi}dx,
\qquad
f(x) = 
\frac{1}{2\pi}\int_{\mathbb{R}^2} \widehat{f}(\xi)e^{i x\cdot \xi} \,d\xi,
\end{equation}
extended to $L^2$ by density. With this convention, Plancherel's Theorem holds: $\|f\|_{2}=\|\widehat f\|_{2}$. 
Similarly, if $\phi \in L^2(\R)$ we denote 
\[
\widehat{\phi}(\omega) = \frac{1}{\sqrt{2 \pi}} \int_\R \phi(x) e^{- i x \cdot \omega} dx,
\quad \phi(x) = \frac{1}{\sqrt{2 \pi}} \int_\R \widehat{\phi}(\omega) e^{i x \cdot \omega} d \omega.
\]

A critical aspect of our analysis is the use of Gevrey functions. We now recall their definition. 
\begin{definition}[Gevrey--\(s\) functions (\cite{Hor2,Rod1})]
For a fixed \(s>1\) and subset \(K \subseteq \R^d\), a $C^\infty$ function $\phi : K \rightarrow \R$ is \emph{Gevrey--$s$} on \(K\) if there exist constants $A,B > 0$ so that for all multi--indices $\alpha \in \Z_{\geq 0}^d$, 
\begin{equation}
\| \partial^\alpha \phi \|_{L^\infty(K)} \leq A B^{|\alpha|} (\alpha !)^s
\end{equation}
where $|\alpha| = \sum_j \alpha_j$ is the order of $\alpha$, and $\alpha! = \prod_j \alpha_j !$. We refer to $(A,B)$ as the Gevrey--$s$ constants of the function $\phi$, and \(K\) as its domain.
\end{definition}
If we do not specify the domain \(K\) in the above definition, then it is assumed to be \(\R^d\).

\begin{definition}\label{def:Gevrey:scale}
A family of functions $\{f_\nu\}$ in a Hilbert space $\cH$ is a \emph{frame} for $\cH$ provided that 
\[
A \| f \|^2 \leq \sum_\nu |\langle f, f_\nu \rangle |^2 \leq B \| f \|^2 \mbox{ for all } f \in \cH.
\] 
Here, $0 < A \leq B < \infty$ are called frame constants for $\{ f_\nu \}$. 
If $\| f_\nu \| = 1$ for all $\nu$ then $\{f_\nu\}$ is called a \emph{unit norm frame}. 
If $A=B$ then $\{f_\nu\}$ is called a \emph{tight frame}. 
A tight frame is \emph{Parseval} if $A=B=1$. 
\end{definition}
In this language, an orthonormal basis for $\cH$ is a unit norm Parseval frame for $\cH$. We use the following elementary result from frame theory, which can be found in classical texts on the subject; for example, see \cite{grochenig2001foundations,christensen2003introduction}. 
\begin{lemma}\label{lem:frame_theory}
Suppose $T : \cH_1 \rightarrow \cH_2$ is a bounded invertible linear operator and $T^{-1} : \cH_2 \rightarrow \cH_1$ is bounded. If $\{f_\nu\}$ is a frame for $\cH_1$ with frame constants $(A,B)$, then $\{ T f_\nu \}$ is a frame for $\cH_2$ with frame constants $(\| T^{-1} \|^{-2} A,\| T \|^2 B)$.
\end{lemma}

\section{Wave packets on the disk}\label{wave-packet-section}

In this section, we collect the technical and geometrical ingredients that we need to build the wave packets on the disk $D(R)$ of radius $R>1$. 
We first break the domain $D(R)$ into manageable pieces by a radial--angular sectorization (\S\ref{sec:sector}), and we label
sectors according to whether they are far from the boundary or close to the boundary of $D(R)$ (interior sectors vs. boundary sectors).
Next, we introduce the smooth (Gevrey) cutoff functions adapted to this partition (\S\ref{sec:cutoffs}); these cutoffs are what make the
later Fourier localization estimates work cleanly.
With the geometry and cutoffs in place, we then define the disk--adapted wave packets (\S\ref{defn:WP}) and prove that they form a
unit norm frame for $L^2(D(R))$ (\S\ref{subsec:frame_property}, Proposition~\ref{prop2}).
After this, the remaining task is essentially quantitative: we show that most of the wave packets are well concentrated in
frequency, and then turn that concentration into eigenvalue counts.

\subsection{A Whitney--type partition of the disk
%\(D(R)\)
} \label{sec:sector} 
We proceed with a Whitney--type decomposition of the disk \(D(R)\) in radial and angular coordinates. 
Without loss of generality, we assume throughout that $R \geq 2$ is dyadic. Thus, $R=2^{j_{\max}}$ for an integer $j_{\max} \geq 1$. See the comment following Theorem~\ref{thm:newmain}. %See Remark \ref{reduction}. 

For $0\le j\le j_{\max}$, partition $S^1=[0,2\pi)$ into $m(j):=2^{j_{\max}-j}$ equal arcs
\[
\Theta_{j,k} := \Big[\tfrac{2\pi(k-1)}{m(j)},\tfrac{2\pi k}{m(j)}\Big)
\quad \textrm{for} \quad 
k=1,\dots,m(j).
\]
For $j<0$, fix $m(j):=2^{j_{\max}}$ and define $\Theta_{j,k} := \Theta_{0,k}$ for $k=1,\dots,m(j)$. Define the \emph{sectors} 
\begin{equation}\label{eqn:sector1}
S_{j,k}:=\{x \in \R^2 : R-2^j\le |x|\le R-2^{j-1}, \arg(x) \in\Theta_{j,k}\}, \quad j \le j_{\max}, 1 \leq k \leq m(j).
\end{equation}
Up to sets of Lebesgue measure zero, \(\{S_{j,k}\}\) forms a decomposition of \(D(R)\); any two distinct sectors intersect on a set of Lebesgue measure zero.

We classify the sectors into two types: 
We call  \(S_{j,k}\) an \textbf{interior sector} if \(j\ge 0\) 
(hence $S_{j,k}$ is contained in the closure of $D(R-\tfrac12)$),  and a \textbf{boundary sector} if \(j<0\) (hence $S_{j,k}$ is contained in $ D(R) \setminus D(R-\tfrac12)$, i.e., near the boundary). 
See Fig. \ref{fig:radialDecomp}. 
On interior sectors we use linear--phase wave packets; on boundary sectors, we use nonlinear--phase wave packets (see \S \ref{defn:WP}).

\begin{figure}[!t]
\centering
\includegraphics[width=.28\linewidth]{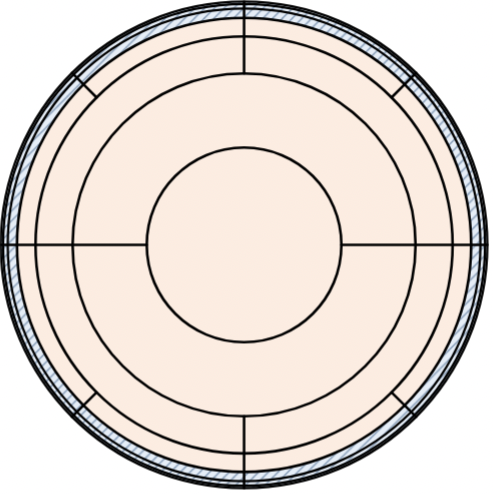}
\caption{Radial–angular sectorization of $D(R)$. Beige: interior sectors; Blue-Gray: boundary sectors. Boundary sectors have a fixed angular width.}
\label{fig:radialDecomp}
\end{figure}

\subsection{Radial and Angular Cutoff Functions}\label{sec:cutoffs}

Fix $s>1$. Choose a family of non--negative radial cutoffs \(\{\phi_j\}_{j\le j_{\max}}\subset C^\infty([0,R])\) such that
\begin{itemize}
\item[(R1)] 
$\sum_{j\le j_{\max}} \phi_j(r)^2=1$ for all $r \in [0,R)$,
\item[(R2)] 
$\supp(\phi_j) \subset I_j:=\big[R-1.12^j, R-0.92^{j-1}\big]\cap[0,R]$,
\item[(R3)] $\| \partial^k_r \phi_j \|_{L^\infty} \leq C_1 C_2^k (k!)^s 2^{-jk}$ for all $k \geq 0$,
\end{itemize} 
where $C_1,C_2 > 0$ are absolute constants depending on $s$ but not on $j$. For the construction of these cutoffs, see Lemma \ref{lem:radial-cutoffs} in Appendix \ref{cutoffs}. 
Condition (R2) implies that only nearest--neighbour overlaps occur:
\[
\supp(\phi_j) \cap \supp(\phi_{j'}) \neq \varnothing
\iff |j-j'|\le 1.
\]
Furthermore, the properties (R1) and (R2) imply that
\begin{equation}\label{eqn:radcutoff1}
\phi_{j_{\max}}(r) = 1 \mbox{ for } r \in [0,R/4]
\end{equation}
and
\[
\phi_j(r) = 1 \mbox{ for } r \in I_j^0 := [R - 0.8 \cdot 2^j, R - 1.2 \cdot 2^{j-1}], \quad j < j_{\max}.
\]
To see this, it suffices to note that $\supp(\phi_{j'})$ is disjoint from $I_j^0$ when $j' \neq j$ due to (R2). Thus, $\phi_j \equiv 1$ on $I_j^0$ due to (R1). Furthermore, (R2) implies that for $j < j_{\max}$, $\supp(\phi_j) \subset [0.45  R, R]$. Therefore, (R1) implies that $\phi_{j_{\max}} \equiv 1$ on $[0, 0.45 R]$.

For each $0\le j\le j_{\max}$, choose a family of non--negative angular cutoffs $\{\eta_{j,k}\}_{1 \leq k \leq m(j)}\subset C^\infty(S^1)$ with
 
\begin{itemize}
\item[(A1)] 
$\sum_{k=1}^{m(j)}\eta_{j,k}(\theta)^2=1$ for all \(\theta \in S^1\),
\item[(A2)] 
$\supp(\eta_{j,k})\subset \Theta_{j,k}^\ast$ where $\Theta_{j,k}^\ast := \Big[\frac{2\pi(k-1)}{m(j)}-0.05\frac{2\pi}{m(j)}, \frac{2\pi k}{m(j)}+0.05\frac{2\pi}{m(j)} \Big)$,
\item[(A3)] $\| \partial^k_\theta \eta_{j,k} \|_{L^\infty} \leq C_1 C_2^k (k!)^s m(j)^{k}$.
\end{itemize}
Here, we recall that $m(j) = 2^{j_{\max}-j} \geq 1$. We identify the interval $\Theta_{j,k}^\ast$ in $\R$ with an interval in $S^1$ through the periodization map $\R \mapsto S^1 = \R /(2 \pi \Z)$. For the construction of such cutoffs, see Lemma \ref{lem:angular-cutoffs} in Appendix \ref{cutoffs}. 

Observe that in the case $j=j_{\max}$, where $m(j) = 1$, condition (A1) implies that
\begin{equation}\label{prop:eta}
\eta_{j_{\max},1} \equiv 1 \mbox{ on } S^1.
\end{equation}
Also, observe that for any $j \le j_{\max}$ and $1 \leq k \leq m(j)$, $\Theta_{j,k}^\ast$ contains $\Theta_{j,k}$ and 
\[
|\Theta_{j,k}^\ast|=1.1|\Theta_{j,k}| = 1.1 \frac{2 \pi}{m(j)}.
\]

\subsection{Wave packet systems on \(L^2(D(R))\)}\label{defn:WP}

For \(0\le j\le j_{\max}\), $1 \leq k \leq m(j)$, $\mathbf{m} = (m_1,m_2) \in \Z^2$, define the interior wave packets (with linear phase)
\begin{equation}\label{eq:type1wave}
\psi_{j,k,\mathbf{m}}(x)
= C_{j,k} 2^{-j}\phi_j(r)\eta_{j,k}(\theta)
e^{i c_\ast 2^{-j} (m_1 x_1 + m_2 x_2)},
\end{equation}
where $r = |x|$ and $\theta=\operatorname{arg}(x)$ are the polar coordinates of $x$, and \(c_\ast>0\) is sufficiently small, to be determined later. 
For \(j<0\), $1 \leq k \leq m(j)$, $\mathbf{m} \in \Z^2$, define the boundary wave packets (with nonlinear phase)
\begin{equation}\label{eq:type2wave}
\psi_{j,k,\mathbf{m}}(x)
= C_{j,k} 2^{-j/2}\phi_j(r)\eta_{0,k}(\theta)
e^{i(1/2) (m_1 2^{-j} r + m_2 R \theta)}.
\end{equation}
We define $c_\ast$ in \eqref{defn:cstar} below. 
Interior wave packets have a plane--wave phase $\psi(x)= i c_\ast 2^{-j} (m_1 x_1 + m_2 x_2)$ that is linear in $x$, while for boundary wave packets, the phase is linear in polar coordinates $(r,\theta)$ but nonlinear in Cartesian coordinates. 

The constants $C_{j,k}$ are picked to ensure that $\| \psi_{j,k,\mathbf{m}} \|_{L^2} = 1$ and satisfy
\begin{equation}\label{eqn:const_bds}
 0 < c \leq C_{j,k} \leq C
\end{equation}
with $c,C$ absolute constants. To see that these normalization constants are uniformly bounded above and below it suffices to note that the measure of the support of the interior wave packet in \eqref{eq:type1wave} is $\approx 2^{2j}$ for $0 \leq j \leq j_{\max}$, and the measure of the support of the boundary wave packet in \eqref{eq:type2wave} is $\approx 2^{j}$ for $j < 0$. Meanwhile, the cutoff function $\phi_j(r) \eta_{j,k}(\theta)$ has $L^\infty$ norm $\leq 1$ while $\phi_j(r) \eta_{j,k}(\theta)$ is $\approx 1$ on a set of proportional measure to its support. The condition \eqref{eqn:const_bds} follows easily from these assertions.

Denote the indexing set for the wave packets by
\begin{equation}\label{eqn:indexset}
\cI = \{ \nu = (j,k,\mathbf{m}) : j \leq j_{\max}, 1 \leq k \leq m(j) := \min\{ 2^{j_{\max} - j},2^{j_{\max}}\}, \mathbf{m} \in \Z^2\}.
\end{equation}
Split $\mathcal I$ into two disjoint subsets,
 \[
\Iint := \{\nu \in \mathcal I : 0 \le j \le j_{\max}\}, 
\qquad
\Ibdry := \{\nu \in \mathcal I : j < 0\}.
\]
Using \eqref{eqn:sector1}, we see that each sector $S_{j,k}$ sits inside the enlarged sector 
\[
S_{j,k}^\ast := \{ x \in \R^2 : |x| \in I_j=[ R- 1.1 \cdot 2^j, R - 0.9 \cdot 2^{j-1}], \arg(x) \in \Theta_{j,k}^\ast\}.
\]
Due to the support properties of the radial and angular cutoff functions (conditions (R2) and (A2) in Section \ref{sec:cutoffs}), we have 
\begin{equation}\label{wp_support}
\supp(\psi_{j,k,\mathbf{m}}) \subset S_{j,k}^\ast \subset D(R).
\end{equation}
In particular, $\{\psi_{j,k,\mathbf{m}}\}$ is a family of $C^\infty$ functions supported on $D(R)$.

\subsection{Frame property of the wave packet family}\label{subsec:frame_property}

Recalling the definition of frames in Section \ref{sec:prelim}, we prove the following result. 

\begin{proposition}\label{prop2}
For an absolute constant \(c_\ast \in (0,1)\) , the family \(\{\psi_{j,k,\mathbf{m}}\}\) forms a unit norm frame for \(L^2(D(R))\) with frame bounds $0 < A < B < \infty$ that are absolute constants independent of $s$ and $R$.
\end{proposition}

\begin{remark}[Independence from the Gevrey parameter]
 The frame bounds in Proposition~\ref{prop2} are uniform in the Gevrey index $s$, since 
the proof uses only the exact partition of unity
$\sum_{j,k}\phi_j(r)^2\eta_{j,k}(\theta)^2=1$
and the orthonormality of the exponential basis on each bounding box.
\end{remark}

\begin{proof} 
According to conditions (R1) and (A1) in Section \ref{sec:cutoffs},
\[
\sum_{j,k} \phi_j(r)^2\eta_{j,k}(\theta)^2 = 1,
\]
which implies that for any $f \in L^2(D(R))$,
\begin{equation}\label{eqn:sector_decomp}
\|f\|_{L^2(D(R))}^2
= \int_{D(R)} |f|^2 \!\left(\sum_{j,k} \phi_j(r)^2 \eta_{j,k}(\theta)^2\right)\!dx
= \sum_{j,k} \| f \phi_j \eta_{j,k} \|_{L^2(D(R))}^2,
\end{equation}
where $r=|x| \in [0,R]$ and $\theta = \arg(x) \in [0,2\pi)$.

\medskip
\noindent\emph{Frame property of interior wave packets \((0 \le j \le j_{\max})\).}
The function \(\phi_j(|x|) \eta_{j,k}(\arg(x))\) is supported on a sector
\(S_{j,k}^\ast\subset D(R)\). Observe that the sector $S_{j,k}^\ast$ has radial extent $\Delta r \approx 2^j$ and angular extent $\Delta \theta \approx 2^j/R$. Because $x_1 = r \cos \theta$ and $x_2 = r \sin \theta$ the change in $x_1$ within the sector $S_{j,k}^\ast$ satisfies $\Delta x_1 \lesssim R \Delta \theta + \Delta r \lesssim 2^j$. Similarly, the change in $x_2$ within $S_{j,k}^\ast$ satisfies $\Delta x_2 \lesssim 2^j$. Therefore, $S_{j,k}^\ast$ is contained in a square \(T_{j,k}\) of side length
\(C_\ast 2^j\) for an absolute constant \(C_\ast>1\). We define
\begin{equation}\label{defn:cstar}
c_\ast = 2 \pi/C_\ast.
\end{equation}
Consider the orthonormal Fourier basis for \(L^2(T_{j,k})\): 
\[
\bigl\{ C_\ast^{-1} 2^{-j} e^{2\pi i2^{-j}\langle \mathbf{m},x\rangle / C_\ast} : \mathbf{m}\in\Z^2 \bigr\}. 
\]
 By Plancherel’s theorem,
\[
\begin{aligned}
\| f\phi_j\eta_{j,k} \|^2
&= \sum_{\mathbf{m}\in\Z^2}
\bigl|\bigl\langle f\phi_j\eta_{j,k},
C_\ast^{-1} 2^{-j} e^{2\pi i2^{-j}\langle \mathbf{m},x\rangle / C_\ast}\bigr\rangle\bigr|^2 
\\
&= \sum_{\mathbf{m}\in\Z^2}
\bigl|\bigl\langle f,
C_\ast^{-1} \phi_j\eta_{j,k} 2^{-j} e^{2\pi i2^{-j}\langle \mathbf{m},x\rangle / C_\ast}\bigr\rangle\bigr|^2 
\\
&= C_\ast^{-2} \sum_{\mathbf{m}\in\Z^2} |\langle f, C_{j,k}^{-1} \psi_{j,k,\mathbf{m}}\rangle|^2,
\end{aligned}
\]
where \(\psi_{j,k,\mathbf{m}}\) are the interior wave packets from \eqref{eq:type1wave} with $c_\ast = 2 \pi C_\ast^{-1}$, and $\| \cdot \|$ is the standard $L^2$ norm. Because $0 < c \leq C_{j,k} \leq C$ for absolute constants $c,C$, we obtain
\begin{equation}\label{eqn:type1_frame}
 c_0 \| f \phi_j \eta_{j,k} \|^2 \leq \sum_{\mathbf{m} \in \Z^2} |\langle f, \psi_{j,k,\mathbf{m}} \rangle|^2 \leq C_0 \| f \phi_j \eta_{j,k} \|^2
\end{equation}
for absolute constants $c_0,C_0 > 0$.

\medskip

\noindent\emph{Frame property of boundary wave packets \((j<0)\).}
The radial cutoff \(\phi_j(r)\) is supported on
\[
I_j = \bigl[R-1.1\cdot 2^j, R-0.9\cdot 2^{j-1}\bigr],
\]
where $|I_j| = (0.65) 2^j$, and the angular cutoff \(\eta_{0,k}(\theta)\) is supported on an interval \(\Theta_{0,k}^\ast \subset S^1\) with
\(|\Theta_{0,k}^\ast| = 2.2 \cdot \pi/R \). 
Define the rectangular box \(B_{j,k}=I_j\times \Theta_{0,k}^\ast\) in \((r,\theta)\)--coordinates, and define the enlarged box $\widetilde{B}_{j,k} = \widetilde{I}_j \times \widetilde{\Theta}_{0,k} \supset B_{j,k}$ such that $\widetilde{I}_j \supset I_j$, $\widetilde{\Theta}_{0,k}\supset \Theta_{0,k}^\ast$ are intervals chosen by dilating $I_j$ and $\Theta_{0,k}^\ast$ about their centers, with $|\widetilde{I}_j| = 4 \pi \cdot 2^j$, $|\widetilde{\Theta}_{0,k}| = 4 \pi/R$. Define an exponential system $\cB$ as 
\[
\cB=\{
h_{\bf m}(r,\theta)= (4 \pi)^{-1} \cdot 2^{-j/2} R^{1/2}
e^{i (1/2) (m_1 2^{-j} r + m_2 R \theta)}: {\bf m}=(m_1,m_2)\in\mathbb{Z}^2
\}.
\]
Observe that $\cB$ is the orthonormal Fourier basis on $L^2(\widetilde{B}_{j,k}, drd\theta)$.

Identify $L^2(B_{j,k},dr\,d\theta)$ with the closed subspace of $L^2(\widetilde{B}_{j,k},drd\theta)$ consisting of functions supported in $B_{j,k}$ via extension by zero. Since $\cB$ is an orthonormal basis for $L^2(\widetilde{B}_{j,k},dr\,d\theta)$, we can restrict Parseval's identity to the subspace $L^2(B_{j,k},dr\,d\theta)$ to obtain 
\[
\| f \|^2_{L^2(B_{j,k}, dr d \theta)} = \sum_{{\bf m}\in \mathbb{Z}^2} \bigl| \langle f, h_{\bf m} \rangle_{L^2(B_{j,k}, dr d \theta)} \bigr|^2, \quad f \in L^2(B_{j,k}, dr d \theta).
\]

Consider the Hilbert spaces $\cH_1 = L^2(B_{j,k},dr\,d\theta), \cH_2=L^2(B_{j,k},r\,dr\,d\theta)$. Because $cR \leq r \leq CR$ for $(r, \theta) \in B_{j,k}$, the norms are equivalent
\[ c R \|f\|_{\mathcal H_1}^2\leq \|f\|_{\mathcal H_2}^2 \leq CR \|f\|_{\mathcal H_1}^2.
\]
Thus, the identity operator $I : \cH_1 \rightarrow \cH_2$ satisfies $\| I \|_{\cH_1 \rightarrow \cH_2} \leq (CR)^{1/2}$ and $\|I^{-1} \|_{\cH_2 \rightarrow \cH_1}^{-1} \geq (c R)^{1/2}$. By Lemma \ref{lem:frame_theory}, $\{ h_{\bf m} \}_{{\bf m}\in \mathbb{Z}^2}$ is a frame for $L^2(B_{j,k}, r dr d \theta)$ with frame constants $(cR, CR)$. Writing down this frame inequality and dividing through by $R^{-1}$ gives 
\[
c \|f\|_{L^2(B_{j,k},r\,dr\,d\theta)}^2
\leq 
\sum_{{\bf m}}\bigl|\langle f, R^{-1/2} h_{\bf m}\rangle_{L^2(B_{j,k}, r dr d \theta) }\bigr|^2
\leq 
C \|f\|_{L^2(B_{j,k},r\,dr\,d\theta)}^2 .
\]
Thus, $\{R^{-1/2}h_{\bf m}\}$ is a frame for $L^2(B_{j,k},r\,dr\,d\theta)$.

Apply this to \(f\phi_j\eta_{0,k}\), and use that
\(L^2(B_{j,k},r\,dr\,d\theta)\) is isometric to \(L^2(S_{j,k}^\ast,dx)\) under the
polar change of variables. We obtain
\[
c\|f\phi_j\eta_{0,k}\|_{L^2(S_{j,k}^\ast)}^2
\leq 
\sum_{{\bf m}}
\bigl|\langle f, R^{-1/2}\phi_j\eta_{0,k}h_{\bf m}\rangle_{L^2(S_{j,k}^\ast)}\bigr|^2
\leq 
C\|f\phi_j\eta_{0,k}\|_{L^2(S_{j,k}^\ast)}^2 .
\]
Observe that
\[
R^{-1/2}\phi_j(r)\eta_{0,k}(\theta)h_{\bf m}(r,\theta)
= (4 \pi C_{j,k})^{-1} \psi_{j,k,\mathbf{m}},
\]
where \(\psi_{j,k,\mathbf{m}}\) is the boundary wave packet defined in \eqref{eq:type2wave}. Here, $0 < c \leq C_{j,k} \leq C$ for absolute constants $c,C$. Thus,
\begin{equation} \label{eqn:type2_frame}
c_1 \|f\phi_j\eta_{0,k}\|^2
\leq 
\sum_{{\bf m}}
\bigl|\langle f,\psi_{j,k, \mathbf{m}} \rangle\bigr|^2
\leq 
C_1 \|f\phi_j\eta_{0,k}\|^2
\end{equation}
for absolute constants $c_1,C_1$.

Combining the interior and boundary contributions by adding \eqref{eqn:type1_frame}, \eqref{eqn:type2_frame} over all $j,k$, using \eqref{eqn:sector_decomp}, we conclude that
\(\{\psi_{j,k,\mathbf{m}}\}\) forms a frame for \(L^2(D(R))\) when $c_\ast = 2 \pi /C_\ast$.
\end{proof}
\section{Fourier localization and energy concentration estimates}\label{sec:Fourier localization and energy concentration estimates}
This section is dedicated to the proof of the following result, which will be used to complete the proof of Theorem \ref{thm:newmain}.

\begin{proposition}[Energy concentration of wave packets]\label{prop1}
Fix a dyadic integer $R \geq 2$. The wave packet family \(\{\psi_\nu\}\) satisfies energy concentration estimates of order \( \varepsilon\) on \(D(R)\times S\) with at most \(CR\log(R/ \varepsilon)^{1+2s}\) residual packets. Precisely: for each \( \varepsilon \in (0,1/2)\) there is a partition of the index set \(\cI=\cI_1\cup \cI_2\cup \cI_3\) such that
\[
\sum_{\cI_1}\|\widehat{\psi_\nu}\|_{L^2(S)}^2
+
\sum_{\cI_2}\|\widehat{\psi_\nu}\|_{L^2(\R^2\setminus S)}^2
\le \varepsilon^2,
\]
while \(\#\cI_3 \le C R\log(R/ \varepsilon)^{1+2s}\). Here, $C$ is a constant determined by $s$ and the geometric parameters of $S$.
\end{proposition}

To prove this result we will analyze the concentration properties of the Fourier transforms of wave packets. The estimates will split into cases for the interior wave packets in Section \ref{sec:energy_type1} and for the boundary wave packets in Section \ref{sec:energy_type2}.

\subsection{Localization of interior wave packets $(0\leq j \leq j_{\max})$}\label{sec:energy_type1}

We categorize the interior wave packets into components $\cI_1^{\mathrm{int}}, \cI_2^{\mathrm{int}}, \cI_3^{\mathrm{int}}$ according to the relative position of the frequency vector
$c_\ast 2^{-j} {\bf m}$, ${\bf m}=(m_1, m_2)$
with respect to the frequency domain $S$. If the distance between 
\(2^{-j}c_\ast{\bf m}\)
and $\partial S$ is sufficiently large, then the packets go in $\cI_1^{\mathrm{int}}$ or $\cI_2^{\mathrm{int}}$. Otherwise, the packets go in $\cI_3^{\mathrm{int}}$. We describe the partition in more detail toward the end of this subsection.

 The next lemma asserts that the function $h_{j,k} = \phi_j(|x|) \eta_{j,k}(\arg(x))$ arising as a factor in the formula for the interior wave packets is a Gevrey--$s$ function, with constants determined by the scale $2^j$. (See Appendix \ref{appendix:Gevrey} for terminology and properties of Gevrey functions.) The case $j=0$ follows from composition properties of Gevrey functions. In this case, one can exploit that the coordinate change $x\to (r,\theta)$ is real-analytic away from the origin and the cutoffs $\phi_j$ and $\eta_{j,k}$ satisfy Gevrey--$s$ regularity conditions (see (R3) and (A3)), while the composition with a real--analytic map preserves Gevrey regularity. The proof of the result below, for general $j \geq 0$, follows by a concrete estimation involving the chain rule.

\begin{lemma}\label{derivative--bounds}
The function $h_{j,k}(x) = \phi_j(|x|) \eta_{j,k}(\arg(x))$ for $0 \leq j \leq j_{\max}$ and $1 \leq k \leq m(j) := 2^{j_{\max}-j}$ is a smooth function on $\R^2$, supported on the sector $S_{j,k}^\ast$, satisfying the derivative bounds
\[
| \partial^\alpha_{x} h_{j,k}(x)| \leq C_1 C_2^{|\alpha|} 2^{-j | \alpha|}(\alpha !)^s \qquad (x\in \Bbb R^2),
\]
for any multi--index $\alpha \in \Z_{\geq 0}^2$. Here, $C_1,C_2 >1$ are constants determined only by $s$.
\end{lemma}

 \begin{proof} 
We denote the partial differentiation operator by $\partial_x^\alpha = \partial_{x_1}^{\alpha_1}\partial_{x_2}^{\alpha_2}$. Differentiation in Cartesian and radial coordinates is related by 
\[
\partial_{x_1}
= \cos\theta\,\partial_r
- \frac{\sin\theta}{r}\,\partial_\theta,
\qquad
\partial_{x_2}
= \sin\theta\,\partial_r
+ \frac{\cos\theta}{r}\,\partial_\theta \quad (r > 0),
\]
where $(r,\theta) = (|x|, \arg(x))$.
We claim that 
 \begin{align}\label{partial--diff}
\partial_x^{\alpha}
= \sum_{ p+q \leq |\alpha|}
P_{\alpha,p,q}(\theta) r^{p-|\alpha|}  \partial_r^{p}\partial_\theta^{q},
\end{align}
where $P_{\alpha,p,q}(\theta)$ is a trigonometric polynomial of degree at most $|\alpha|$ having the form
\begin{equation}\label{eqn:trigpoly}
P_{\alpha,p,q}(\theta) = \sum_{k=-|\alpha|}^{|\alpha|} c_{\alpha,p,q}(k) e^{i k \theta}, \quad 
 \text{with } 
 |c_{\alpha,p,q}(k)| \leq 4^{|\alpha|} |\alpha|^{|\alpha|-p-q}.
\end{equation}
Once we establish this claim we will know that 
\begin{align}\label{ineq:P}
\| P_{\alpha,p,q} \|_{L^\infty} \leq (2|\alpha|+1) 4^{|\alpha|} |\alpha|^{|\alpha|-p-q}\leq C^{|\alpha|}|\alpha|^{|\alpha|-p-q}.
\end{align}

We prove \eqref{partial--diff}--\eqref{eqn:trigpoly} by induction on $\alpha$. When $\alpha = 0$ the claim is true with $P_{0,0,0} = 1$. Suppose that the claim is true for some multindex $\alpha$. To complete the induction step we will establish the claim for the multi--index $\beta = \alpha + (1,0)$. (We omit the proof for the multi--index $\alpha + (0,1)$ as it is similar.) Then
\[
\begin{aligned}
\partial^\beta_x & = \partial_{x_1}( \partial^\alpha_x) = \sum_{p+q \leq |\alpha|} \partial_{x_1} \left(P_{\alpha,p,q}(\theta) r^{p-|\alpha|} \partial_r^{p}\partial_\theta^{q}\right)\\
&= \sum_{p+q \leq |\alpha|} \cos \theta \cdot \partial_r \left[ P_{\alpha,p,q}(\theta) r^{p-|\alpha|} \partial_r^p \partial_\theta^q \right] - \sum_{p+q \leq |\alpha|} \frac{\sin \theta}{r} \partial_\theta \left[ P_{\alpha,p,q}(\theta) r^{p-|\alpha|} \partial_r^p \partial_\theta^q \right] \\
&= \sum_{p+q \leq |\alpha|} \left[ (p-|\alpha|) \cos \theta \cdot P_{\alpha,p,q}(\theta) - \sin \theta \cdot P'_{\alpha,p,q}(\theta)\right] r^{p-|\alpha|-1} \partial_r^p \partial_\theta^q \\
& \qquad + \sum_{p+q \leq |\alpha|} \cos \theta \cdot P_{\alpha,p,q}(\theta) r^{p-|\alpha|} \partial_r^{p+1} \partial_\theta^q - \sum_{p+q \leq |\alpha|} \sin \theta \cdot P_{\alpha,p,q}(\theta) r^{p-|\alpha|-1} \partial_r^{p} \partial_\theta^{q+1}.
\end{aligned}
\]
We modify the index of summation in the second and third sums in the last line, and recall $|\beta| = |\alpha| + 1$, so,
\[
\begin{aligned}
\partial^\beta_x &= \sum_{p+q \leq |\beta| - 1 } \left[ (p-|\alpha|) \cos \theta \cdot P_{\alpha,p,q}(\theta) - \sin \theta \cdot P'_{\alpha,p,q}(\theta)\right] r^{p-|\beta|} \partial_r^p \partial_\theta^q \\
& + \sum_{p+q \leq |\beta|, p>0} \cos \theta \cdot P_{\alpha,p-1,q}(\theta) r^{p-|\beta|} \partial_r^{p} \partial_\theta^q - \sum_{p+q \leq |\beta|, q>0} \sin \theta \cdot P_{\alpha,p,q-1}(\theta) r^{p-|\beta|} \partial_r^{p} \partial_\theta^{q}.
\end{aligned}
\]
Therefore,
\[
\partial^\beta_x = \sum_{p+q \leq |\beta|} P_{\beta,p,q}(\theta) r^{p-|\beta|} \partial_r^p \partial_\theta^q
\]
where
\begin{equation}\label{eqn:trigpoly_ind}
\begin{aligned}
P_{\beta,p,q}(\theta) &:= \mathbf{1}_{p+q\leq |\beta|-1} \left[ (p-|\alpha|) \cos \theta \cdot P_{\alpha,p,q}(\theta) - \sin \theta \cdot P'_{\alpha,p,q}(\theta) \right] 
\\
&\qquad+ \mathbf{1}_{p > 0} \cos(\theta) \cdot P_{\alpha,p-1,q}(\theta) - \mathbf{1}_{q > 0} \sin(\theta) \cdot P_{\alpha,p,q-1}(\theta).
\end{aligned}
\end{equation}
According to the inductive hypothesis, $P_{\alpha,p,q}(\theta)$ is a $2 \pi$-periodic trigonometric polynomial of degree at most $|\alpha|$ with coefficients bounded by $4^{|\alpha|} |\alpha|^{|\alpha|-p-q}$ -- see \eqref{eqn:trigpoly}. It follows from \eqref{eqn:trigpoly_ind} that $P_{\beta,p,q}(\theta)$ is a $2\pi$--periodic trigonometric polynomial of degree at most $|\beta| = |\alpha|+1$ with coefficients of size bounded as
\[
\begin{aligned}
|c_{\beta,p,q}(k)| &\leq 2|\alpha| \cdot 4^{|\alpha|} |\alpha|^{|\alpha|-p-q} + 2 \cdot 4^{|\alpha|} |\alpha|^{|\alpha| - p - q + 1} = 4^{|\alpha|+1} |\alpha|^{|\alpha|-p-q+1} \leq 4^{|\beta|} |\beta|^{|\beta|-p-q}.
\end{aligned}
\]
(Here we use that the Fourier coefficients of $P'_{\alpha,p,q}$ are bounded by $|\alpha|$ times the corresponding Fourier coefficients of $P_{\alpha,p,q}$.) This completes the proof of \eqref{partial--diff}--\eqref{eqn:trigpoly} by induction on $\alpha$.

We shall now establish the derivative bound in the statement of the lemma. The bound is obviously true for $\alpha = 0$. For $\alpha \neq 0$, by \eqref{partial--diff} we have 
\begin{equation}\label{eqn:expression1}\partial_{x}^{\alpha} h_{j,k}(x) = \sum_{p+q \leq |\alpha|}
P_{\alpha,p,q}(\theta) r^{p-|\alpha|} \bigl(\partial_r^{p}\phi_j(r)\bigr)\bigl(\partial_{\theta}^{q}\eta_{j,k}(\theta)\bigr).
\end{equation}

We claim that $\partial_x^\alpha h_{j,k}$ is supported on the annular region $R/4 \leq |x| \leq R$. Indeed, if $0 \leq j < j_{\max}$ then condition (R2) of the radial cutoff (see Section \ref{sec:cutoffs}) implies $\supp(\phi_j) \subset [R-1.1 \cdot 2^j, R] \subset [R/4,R]$, proving the claim. On the other hand, if $j = j_{\max}$ then $\eta_{j_{\max},1} \equiv 1$ (see \eqref{prop:eta}) while $\phi_{j_{\max}} \equiv 1$ on $[0,R/4]$ (see \eqref{eqn:radcutoff1}) -- hence, $h_{j_{\max},1} \equiv 1$ for $|x| \leq R/4$, implying that $\partial^\alpha_x h_{j_{\max},1} \equiv 0$ for $|x| \leq R/4$, proving the claim.

Conditions (R3) and (A3) of the radial cutoff $\phi_j$ and angular cutoff $\eta_{j,k}$ assert that
\begin{equation}\label{eqn:cutoff_Gev}
\left\{
\begin{aligned}
&|\partial_r^{p}\phi_j(r)| \leq C_1 C_2^p (p!)^s 2^{-jp} \qquad\quad (j \leq j_{\max}, p \geq 0)\\
&|\partial_\theta^q \eta_{j,k}(\theta)| \leq C_1 C_2^q (q!)^s m(j)^{q} \qquad (0 \leq j \leq j_{\max}, q \geq 0)
\end{aligned}
\right.
\end{equation}
where $m(j) = 2^{j_{\max} - j} = R 2^{-j}$ for $j \geq 0$.

By the previous claim, $R/4\le r\le R$ on $\supp(\partial^\alpha_x h_{j,k})$, 
and since $p-|\alpha|\le 0$ for $p+q\le|\alpha|$, we can bound $r^{p-|\alpha|}\le (R/4)^{p-|\alpha|}$ in \eqref{eqn:expression1}. Using also \eqref{eqn:cutoff_Gev} in \eqref{eqn:expression1}, by the triangle inequality,
\[
|\partial_x^\alpha h_{j,k}(x)| \leq \sum_{p+q \leq |\alpha|} |P_{\alpha,p,q}(\theta)| (R/4)^{p-|\alpha|} C_1^2 C_2^{p+q} (p! q!)^s R^q 2^{-j(p+q)}.
\]
Given that $\| P_{\alpha,p,q}\|_{L^\infty} \leq C^{|\alpha|} |\alpha|^{|\alpha|-p-q}$ (see \eqref{ineq:P}) and the bound $C_2^{p+q} \leq C_2^{|\alpha|}$, for some $A,B >1$ we have
\[
|\partial_x^\alpha h_{j,k}(x)| \leq A \cdot B^{|\alpha|} \sum_{p+q \leq |\alpha|} R^{p + q - |\alpha|} 2^{-j(p+q )} |\alpha|^{|\alpha|-p-q} (p! q!)^s.
\]
Note that $R \geq 2^j$ for $0 \leq j \leq j_{\max}$, and so $R^{p+q-|\alpha|} \leq 2^{j(p+q-|\alpha|)}$ for $p+q \leq |\alpha|$. Thus,
\[
\begin{aligned}
|\partial_x^\alpha h_{j,k}(x)| & \leq A \cdot B^{|\alpha|} \sum_{p+q \leq |\alpha|} 2^{j(p + q - |\alpha|)} 2^{-j(p+q )} |\alpha|^{|\alpha|-p-q} (p! q!)^s \\
&= A \cdot B^{|\alpha|} \cdot 2^{-j|\alpha|} \sum_{p+q \leq |\alpha|} |\alpha|^{|\alpha|-p-q} (p! q!)^s.
\end{aligned}
\]
Note that the number of terms in the above sum is $(|\alpha|+1)(|\alpha|+2)/2\leq (|\alpha|+1)^2 \leq C^{|\alpha|}$ for some absolute $C>1$. Moreover, for $p+q\leq |\alpha|$, we have $p! \leq |\alpha|^p$, $q!\le |\alpha|^{q}$ and hence,
\[
|\alpha|^{|\alpha|-p-q}(p!q!)^s
\le |\alpha|^{|\alpha|-p-q}|\alpha|^{sp}|\alpha|^{sq}
=|\alpha|^{|\alpha|+(s-1)(p+q)}
\le |\alpha|^{s|\alpha|}.
\]
Finally, using Stirling's approximation and the multinomial bound (in dimension $2$),
\[
|\alpha|^{|\alpha|}\le e^{|\alpha|} |\alpha|! \le (2e)^{|\alpha|}\alpha!.
\]
Therefore,
\[
\sum_{p+q\le|\alpha|} |\alpha|^{|\alpha|-p-q}(p!q!)^s \leq \sum_{p+q\le|\alpha|} |\alpha|^{s|\alpha|} \le (|\alpha| + 1) (2e)^{s |\alpha|}(\alpha!)^s
\le C^{|\alpha|}(\alpha!)^s,
\]
and hence,
\[
|\partial_x^\alpha h_{j,k}(x)|
\le C_3 C_4^{|\alpha|} 2^{-j|\alpha|} (\alpha!)^s.
\]
 \end{proof}
The following lemma states that the Fourier transform of a wave packet $\psi_{j,k,\mathbf{m}}$ that is interior (i.e., with $0 \leq j \leq j_{\max}$) is tightly concentrated about the frequency vector $c_\ast 2^{-j} \mathbf{m}$. 
\begin{lemma}\label{lem:typeI}
For any $\mathbf{m}\in \Z^2$, $0 \leq j \leq j_{\max}$, $1 \leq k \leq m(j) = 2^{j_{\max} - j}$, 
\[
 |\widehat{\psi_{j,k,\mathbf{m}}}(\xi)| \leq C 2^{j} \exp(- c |2^j \xi - c_\ast\mathbf{m}|^{1/s}) \qquad (\xi \in \R^2),
\]
where $C,c > 0$ are constants determined by $s$.
\end{lemma}
 \begin{proof} 
Observe $\psi_{j,k,\mathbf{m}}(x) = 2^{-j} h_{j,k}(x) \exp(i c_\ast 2^{-j} \langle \mathbf{m}, x \rangle)$ where $h_{j,k}(x) = \phi_j(|x|) \eta_{j,k}(\arg(x))$ is a smooth function on $\R^2$, supported on the sector $S_{j,k}^\ast$ of measure $\approx 2^{2j}$. Also, Lemma \ref{derivative--bounds} asserts that
\[
\| \partial^\alpha_{x} h_{j,k} \|_{L^\infty} \leq C_1 C_2^{|\alpha|} (\alpha !)^s (2^{-j})^{|\alpha|},
\]
so that $h_{j,k}$ is Gevrey--$s$ with constants $(C_1,C_2 2^{-j})$. Lemma \ref{lemma:Gevrey:Fourier_decay} then implies that
\[
|\widehat{h_{j,k}}(\xi)| \leq C C_1 2^{2j} \exp( - c | 2^j \xi/C_2|^{1/s}),
\]
where $C=e$ and $c=\frac{1}{2e}$. Thus, for constants $C_3,c_3 > 0$ determined by $s$,
\[
|\widehat{h_{j,k}}(\xi)| \leq C_3 2^{2j} \exp( - c_3 | 2^j \xi|^{1/s}).
\]
Finally, $\widehat{\psi_{j,k,\mathbf{m}}}(\xi)= 2^{-j} \widehat{h_{j,k}}(\xi - c_\ast 2^{-j} \mathbf{m})$ and we obtain the desired result. 
\end{proof}
We categorize the interior wave packets into components $\cI_1^{\mathrm{int}}, \cI_2^{\mathrm{int}}, \cI_3^{\mathrm{int}}$. Let $\delta \in (0,1/2)$ be given. We fix a large parameter $A > 1$, and define index sets for the threshold $\delta$,
\begin{equation}\label{defn:partition_int}
% \left\{
\begin{aligned}
&\cI_1^{\mathrm{int}} = \{ (j,k,\mathbf{m}) : 0 \leq j \leq j_{\max}, 1 \leq k \leq 2^{j_{\max}-j}, \mathrm{dist}(c_\ast2^{-j}\mathbf{m}, S) \geq A 2^{-j} \log(1/\delta)^s \},\\
&\cI_2^{\mathrm{int}} = \{ (j,k,\mathbf{m}) : 0 \leq j \leq j_{\max}, 1 \leq k \leq 2^{j_{\max}-j}, \mathrm{dist}(c_\ast2^{-j}\mathbf{m}, \R^2 \setminus S) \geq A 2^{-j} \log(1/\delta)^s \},\\
&\cI_3^{\mathrm{int}} = \{ (j,k,\mathbf{m}) : 0 \leq j \leq j_{\max}, 1 \leq k \leq 2^{j_{\max}-j}, \mathrm{dist}(c_\ast2^{-j}\mathbf{m}, \partial S) \leq A 2^{-j} \log(1/\delta)^s \}.
\end{aligned}
% \right.
\end{equation}
Evidently, the index set for the interior wave packets is decomposed by $\cI_1^{\mathrm{int}}, \cI_2^{\mathrm{int}}, \cI_3^{\mathrm{int}}$, that is,
\[
\cI^{\mathrm{int}} = \{(j,k, \mathbf{m}) : 0 \leq j \leq j_{\max}, 1 \leq k \leq 2^{j_{\max} - j} , \mathbf{m} \in \Z^2 \} = \cI_1^{\mathrm{int}} \cup \cI_2^{\mathrm{int}} \cup \cI_3^{\mathrm{int}} .
\]
First, we obtain an estimate for the cardinality of $\cI_3^{\mathrm{int}}$.

\begin{lemma}
\label{lem:residual_card_bd} 
\begin{equation}\label{eqn:res_bound1}
\#(\cI_3^{\mathrm{int}}) \lesssim_{A,S} R \max\{  \log R \log(1/\delta)^{s},  \log(1/\delta)^{2s} \}.
\end{equation}
\end{lemma}
\begin{proof}
For fixed $0 \leq j \leq j_{\max}$, let
\begin{equation}
\label{eqn:lattice_count1}
\begin{aligned} 
M_j &:= \#\{ \mathbf{m} \in \Z^2 : \mathrm{dist}(c_\ast2^{-j}\mathbf{m}, \partial S) \leq A 2^{-j} \log(1/\delta)^s \} \\
&= \#\{ \mathbf{m} \in \Z^2 : \dist(\mathbf{m}, c_\ast^{-1}2^j \partial S) \leq A c_\ast^{-1} \log(1/\delta)^s \} \\
& \leq C_{S} \max\{ A c_\ast^{-1} \log(1/\delta)^s \cH^1( c_\ast^{-1}2^j \partial S), (A c_\ast^{-1} \log(1/\delta)^s)^2 \} \\
& \leq C_{A,S} \max\{ 2^j \log(1/\delta)^s, \log(1/\delta)^{2s}\}. 
\end{aligned}
\end{equation}
For the proof of the first inequality in the above, we use Lemma \ref{lem:lattice_near_dilated_bd} noting that $t = A c_\ast^{-1} \log(1/\delta)^s$ is at least $1$, together with the fact that $c_\ast^{-1}2^j S$ is well-shaped with the same constant as $S$; see Remark \ref{rem:well-shaped-dilationinvariance}. In the second inequality, we absorb the quantities $\cH^1(\partial S)$, $A$, $c_\ast^{-1}$ into the constant $C_{A,S}$. 
Therefore,
\begin{align*}
\#(I^{\mathrm{int}}_3)
\leq
\sum_{j=0}^{j_{\max}}
2^{j_{\max}-j} M_j
& \lesssim_{A,S}
\sum_{j=0}^{j_{\max}}
2^{j_{\max}-j}\max\{ 2^j \log(1/\delta)^s, \log(1/\delta)^{2s} \} \\
&\lesssim \max\{ j_{\max} 2^{j_{\max}}\log(1/\delta)^{s}, 2^{j_{\max}} \log(1/\delta)^{2s} \}.
\end{align*}
Recalling that $R=2^{j_{\max}}$ and $j_{\max} = \log(R)$, we obtain \eqref{eqn:res_bound1}.
\end{proof}

Now, we state the energy estimates corresponding to this partition:
\begin{lemma}\label{lem:energy_type1}
Suppose $A$ is larger than a big enough constant determined by $s$. Given $\delta \in (0,1/2)$, define the index sets $\cI_1^{\mathrm{int}},\cI_2^{\mathrm{int}}, \cI_3^{\mathrm{int}}$ as in \eqref{defn:partition_int}. Then
\begin{equation}\label{eqn:energy_est}
\sum_{\nu \in \cI_1^{\mathrm{int}}} \| \widehat{\psi_\nu} \|_{L^2(S)}^2 + \sum_{\nu \in \cI_2^{\mathrm{int}}} \| \widehat{\psi_\nu} \|_{L^2(\R^2 \setminus S)}^2 \leq C R^2 \delta^2,
\end{equation}
where $C>0$ is a constant determined by $s$.
\end{lemma}
\begin{proof}
We introduce some notation. For $\mathbf{m} \in \Z^2$ and $j \in \Z$ let $S_j(\mathbf{m})$ denote the $c_\ast2^{-j} \times c_\ast 2^{-j}$ axis--parallel square having $c_\ast 2^{-j} \mathbf{m}$ as its bottom left corner. Then the collection $\{S_j(\mathbf{m})\}_{\mathbf{m} \in \Z^2}$ is a cover of $\R^2$ by closed squares that are pairwise disjoint up to sets of Lebesgue measure zero. Evidently,
\begin{equation}
 \label{eqn:Sjprop}
 |\xi' -c_\ast 2^{-j} \mathbf{m}| \leq \sqrt{2} c_\ast 2^{-j} \mbox{ for all } \xi' \in S_j(\mathbf{m}).
\end{equation}
We estimate the first sum in \eqref{eqn:energy_est} indexed by $\cI_1^{\mathrm{int}}$. Let $\mathcal{H}_{j} \subset \Z^2$ be the set of all $\mathbf{m}$ such that $\dist(c_\ast2^{-j} \mathbf{m}, S) \geq A 2^{-j} \log(1/\delta)^s$, so that 
\[
(j,k,\mathbf{m}) \in \cI_1^{\mathrm{int}} \iff 0 \leq j \leq j_{\max}, 1 \leq k \leq 2^{j_{\max}-j}, \mathbf{m} \in \cH_{j}.
\]
By Lemma \ref{lem:typeI}, 
\begin{equation}\label{eqn:integral1}
\begin{aligned}
 \sum_{(j,k,\mathbf{m}) \in \cI^{\mathrm{int}}_1} \| \widehat{\psi_{j,k,\mathbf{m}}} \|_{L^2(S)}^2 &= \sum_{(j,k,\mathbf{m}) \in \cI^{\mathrm{int}}_1} \int_{S} |\widehat{\psi_{j,k,\mathbf{m}}}(\xi)|^2 d \xi \\
& \lesssim_s \sum_{(j,k,\mathbf{m}) \in \cI_1^{\mathrm{int}}} \int_{S} \left(2^j \exp( - c |2^j \xi - c_\ast \mathbf{m}|^{1/s}) \right)^2 d \xi \\
& = \sum_{j=0}^{j_{\max}} \sum_{k=1}^{2^{j_{\max}-j}} 2^{2j} \int_{S} \sum_{\mathbf{m} \in \cH_{j}} \exp( - 2c |2^j \xi - c_\ast \mathbf{m}|^{1/s}) d \xi.
\end{aligned}
\end{equation}
By \eqref{eqn:Sjprop}, for $\xi \in S$ and $\xi' \in S_j(\mathbf{m})$, $\mathbf{m} \in \cH_j$, by the triangle inequality,
\begin{equation}\label{eqn:sep_est}
\begin{aligned}
|\xi - \xi'| & \geq |\xi - c_\ast 2^{-j} \mathbf{m}| - |\xi'-c_\ast 2^{-j} \mathbf{m}| \geq \dist(c_\ast 2^{-j} \mathbf{m} , S) - \sqrt{2} c_\ast 2^{-j} \\
& \geq A 2^{-j} \log(1/\delta)^s - \sqrt{2} c_\ast 2^{-j} \geq (A/2) 2^{-j} \log(1/\delta)^s, 
\end{aligned}
\end{equation}
provided that $A \geq 2 \sqrt{2} c_\ast$ -- note, in the last inequality we have also used that $\log(1/\delta)^s \geq 1$. 
Since $|2^j \xi' - c_\ast \mathbf{m}|$ is uniformly bounded on $S_j( \mathbf{m})$,
\[
\exp(-2c|2^j \xi - c_\ast \mathbf{m}|^{1/s}) \leq C \exp(-2c|2^j \xi - 2^j \xi'|^{1/s}) \mbox{ for } \xi' \in S_j(\mathbf{m}).
\]
(Here, we exploit the fact that if $|t-s| \leq C$ then $e^t \leq e^C e^s$.)
Therefore, for fixed $\xi \in S$, 
\[
\begin{aligned}
\sum_{\mathbf{m} \in \cH_j} \exp( -2 c |2^j \xi - c_\ast \mathbf{m}|^{1/s}) &\lesssim \sum_{\mathbf{m} \in \cH_j } \frac{1}{|S_j(\mathbf{m})|}\int_{S_j(\mathbf{m})} \exp( -2 c |2^j \xi - 2^j \xi' |^{1/s})d\xi' \\
&= \sum_{\mathbf{m} \in \cH_j } (c_\ast 2^{-j})^{-2} \int_{S_j(\mathbf{m})} \exp(-2c |2^j\xi- 2^j \xi'|^{1/s}) d\xi' \\
& \lesssim \int_{|\eta| \geq (A/2) \log(1/\delta)^s} \exp(-2c |\eta|^{1/s}) d\eta,
\end{aligned}
\]
where in the last inequality we have used that $\{S_j(\mathbf{m})\}_{\mathbf{m} \in \cH_j}$ is pairwise disjoint up to Lebesgue null sets, together with the estimate \eqref{eqn:sep_est} and the change of variables $\eta = 2^j(\xi-\xi')$ in the integral.
Because $t \mapsto \exp(-2c t^{1/s})$ is rapidly decaying, we can continue the estimation to obtain
\begin{equation}\label{eqn:integral2}
\begin{aligned}
& \sum_{\mathbf{m} \in \cH_j} \exp( -2 c |2^j \xi - c_\ast \mathbf{m}|^{1/s}) \lesssim \int_{|\eta| > (A/2) \log(1/\delta)^s} \exp(-2c |\eta|^{1/s}) d\eta \\
& \quad = 2\pi \int_{(A/2) \log(1/\delta)^s}^\infty t \exp(-2c t^{1/s})
dt \lesssim \exp( -c' (A \log(1/\delta)^s)^{1/s})
\end{aligned}
\end{equation}
for an absolute constant $0 < c' < c$.
Combining \eqref{eqn:integral1} and \eqref{eqn:integral2}, we have
\[
\begin{aligned}
\sum_{(j,k,\mathbf{m}) \in \cI^{\mathrm{int}}_1} \| \widehat{\psi_{j,k,\mathbf{m}}} \|_{L^2(S)}^2 & \lesssim_s \sum_{j=0}^{j_{\max}} \sum_{k=1}^{2^{j_{\max}-j}} 2^{2j} \exp(-c' A^{1/s} \log(1/\delta)) \int_S d \xi \\
& = 2^{j_{\max}} \sum_{j=0}^{j_{\max}} 2^{j} \exp(-c' A^{1/s} \log(1/\delta)) |S| \\
& \lesssim 2^{2j_{\max}} \exp(-c' A^{1/s} \log(1/\delta)) \\
& = R^2 \exp(-c' A^{1/s} \log(1/\delta)).
\end{aligned}
\]
where we have used that $2^{j_{\max}} = R$ and the Lebesgue measure of $S$ is bounded by an absolute constant. For large enough $A$ we have $c' A^{1/s} \geq 2$, which implies $\exp(-c' A^{1/s} \log(1/\delta)) \leq \exp(-2 \log(1/\delta)) = \delta^2$, so that
\begin{equation}\label{eqn:energyI1}
\sum_{(j,k,\mathbf{m}) \in \cI^{\mathrm{int}}_1} \| \widehat{\psi_{j,k,\mathbf{m}}} \|_{L^2(S)}^2 \lesssim_s R^2 \delta^2.
\end{equation}

We now estimate the second sum in \eqref{eqn:energy_est} indexed by $\cI_2^{\mathrm{int}}$. Let $\mathcal{M}_{j} \subset \Z^2$ be the set of all $\mathbf{m}$ such that $\dist(c_\ast2^{-j} \mathbf{m}, \R^2 \setminus S) \geq A 2^{-j} \log(1/\delta)^s$, so that 
\[
(j,k,\mathbf{m}) \in \cI_2^{\mathrm{int}} \iff 0 \leq j \leq j_{\max}, 1 \leq k \leq 2^{j_{\max}-j}, \mathbf{m} \in \cM_{j}.
\]
Replicating the previous argument using Lemma \ref{lem:typeI}, we have
\[
\sum_{(j,k,\mathbf{m}) \in \cI_2^{\mathrm{int}}} \| \widehat{\psi_{j,k,\mathbf{m}}} \|_{L^2(\R^2 \setminus S)}^2 \lesssim_s \sum_{j=0}^{j_{\max}} \sum_{k=1}^{2^{j_{\max}-j}} 2^{2j} \int_{\R^2 \setminus S} \sum_{\mathbf{m} \in \cM_{j}} \exp( - 2c |2^j \xi - c_\ast \mathbf{m}|^{1/s}) d \xi.
\]
As in the proof of \eqref{eqn:sep_est}, we have
\begin{equation}
 |\xi - \xi'| \geq (A/2) 2^{-j} \log(1/\delta)^s \mbox{ for all } \xi \in \R^2 \setminus S, \; \xi' \in S_j(\mathbf{m}), \; \mathbf{m} \in \cM_j.
\end{equation}
For fixed $\xi \in \R^2 \setminus S$, the sum inside the integral can be estimated as 
\[
\begin{aligned}
\sum_{\mathbf{m} \in \cM_j} \exp( - 2 c |2^j \xi - c_\ast \mathbf{m}|^{1/s}) \lesssim 2^{2j} \int_{ \dist(\xi', \R^2 \setminus S) > (A/2) 2^{-j} \log(1/\delta)^s} \exp(-c' |2^j \xi - 2^j \xi'|^{1/s}) d\xi'.
\end{aligned}
\]
Hence, \[
\begin{aligned}
&\int_{\R^2 \setminus S} \sum_{\mathbf{m} \in \cM_j} \exp( - 2c |2^j \xi - c_\ast \mathbf{m}|^{1/s}) d \xi \\
&\lesssim 2^{2j} \int_{\R^2 \setminus S} \int_{ \dist(\xi', \R^2 \setminus S) > (A/2) 2^{-j} \log(1/\delta)^s} \exp(-c' |2^j \xi - 2^j \xi'|^{1/s}) d \xi' d \xi \\
&= 2^{-2j} \int_{\R^2 \setminus (2^j S)} \int_{ \dist(x, \R^2 \setminus (2^j S)) > (A/2) \log(1/\delta)^s} \exp(-c' |x - y|^{1/s}) dx d y \\
& \leq 2^{-2j} \int_{(x,y) \in \R^2 \times \R^2 : |x-y| > (A/2) \log(1/\delta)^s} \exp(-c'|x-y|^{1/s}) dx dy \\
& \lesssim 2^{-2j} \exp(-c'' A^{1/s} \log(1/\delta)).
\end{aligned}
\]
In the second to last line, we use the larger domain 
\[
\{(x,y)\in \mathbb{R}^2\times \mathbb{R}^2 : |x-y| > (A/2) \log(1/\delta)^s \}.
\]
This is allowed because
if $(x,y)$ satisfies the original restrictions, then 
$|x-y| > (A/2) \log(1/\delta)^s$.

Therefore,
\[
\begin{aligned}
\sum_{(j,k,\mathbf{m}) \in \cI_2^{\mathrm{int}}} \| \widehat{\psi_\nu} \|_{L^2(\R^2 \setminus S)}^2 
&\lesssim_s 
\sum_{j=0}^{j_{\max}} \sum_{k=1}^{2^{j_{\max}-j}} 2^{2j} 2^{-2j} \exp(-c'' A^{1/s} \log(1/\delta)) \\
& \leq 
2^{j_{\max}+1} \exp(-c'' A^{1/s} \log(1/\delta)).
\end{aligned}
\]
We assume $A$ is sufficiently large so that $c'' A^{1/s} \geq 2$, hence the above implies
\begin{equation}\label{eqn:energyI2}
\sum_{(j,k,\mathbf{m}) \in \cI_2^{\mathrm{int}}} \| \widehat{\psi_\nu} \|_{L^2(\R^2 \setminus S)}^2 \lesssim_s 2^{j_{\max}} \delta^2 
= R \delta^2.
\end{equation}
Combining \eqref{eqn:energyI1} and \eqref{eqn:energyI2} gives the desired result.
\end{proof}

\subsection{Localization of boundary wave packets $(j<0)$.}\label{sec:energy_type2}

All boundary wave packets are either placed in \(\cI_1\) (high frequency, small norm in \(S\)) or \(\cI_3\) (residual). The goal of this section is to prove Lemma \ref{lem:type2_decay} below, which establishes the decay of the Fourier transform of boundary wave packets. We will then use these bounds for categorization. 

We will require a formula for the Fourier transform of a tensor product $u(r,\theta) = f(r) g (\theta)$, where $(r,\theta) = (|x|, \arg(x))$ are radial coordinates for the spatial variable $x$. If $(\rho, \phi)$ denote radial coordinates for the frequency variable $\xi$, then the Fourier transform $\hat u(\xi) = \frac{1}{2\pi} 
 \int_{\mathbb{R}^2} u(x) e^{- i\xi\cdot x} dx$ satisfies
$$
\hat u(\rho, \phi) = \frac{1}{2\pi} \int_0^{2\pi} \int_0^\infty u(r,\theta) e^{-i r\rho\cos(\theta-\phi)} rd\theta dr. 
$$
By the Jacobi-Anger expansion 
$$e^{-i r\rho\cos(\theta-\phi)}=\sum_{n\in\mathbb{Z}} i^{-n}J_n(r\rho) e^{in(\theta-\phi)},$$
where $J_n$ is the Bessel function of integer order $n$. So, we obtain
\begin{equation}
\label{Bessel-expansion}
\widehat{u}(\rho,\phi) = \frac{1}{2\pi} 
 \sum_{n \in \Z} i^{-n} e^{-in \phi} \left[ \int_0^\infty f(r) J_n(r \rho) r dr \right] \left[ \int_0^{2 \pi} g(\theta) e^{i n \theta} d \theta \right].
\end{equation}

When \(j<0\) the boundary wave packets are defined in \eqref{eq:type2wave}; by using the expansion \eqref{Bessel-expansion},
\[
\begin{aligned}
\widehat{\psi_{j,k,\mathbf{m}}}(\xi)
&= \frac{1}{2\pi} \sum_{n\in\Z} i^{-n} e^{-in\phi}
\Bigl[
\int_0^\infty 2^{-j/2} \phi_j(r) e^{i m_1 2^{-j-1} r}
 J_n(r\rho) r dr
\Bigr]
\Bigl[
\int_0^{2\pi} \eta_{0,k}(\theta) e^{i( m_2 R/2 + n)\theta}\,d\theta
\Bigr] \\
&= \frac{1}{2\pi} 
 \sum_{n\in\Z} i^{-n} e^{-in\phi} \Psi^{\mathrm{rad}}_{j,m_1,n}(\rho) \cdot \Psi^{\mathrm{ang}}_{k,m_2,n} \quad (\rho = |\xi|, \phi = \arg(\xi)),
\end{aligned}
\]
where 
\begin{align}\label{rad}
\Psi^{\mathrm{rad}}_{j,m_1,n}(\rho)
&:= \int_0^\infty 2^{-j/2} r \phi_j(r) e^{i m_1 2^{-j-1} r} J_n(r\rho)\,dr,
\\\label{ang}
\Psi^{\mathrm{ang}}_{k,m_2,n}
&:= \int_0^{2\pi} \eta_{0,k}(\theta) 
e^{i(m_2 R/2 + n)\theta}\,d\theta.
\end{align}

We first consider the bounds on the angular factor $\Psi^{\mathrm{ang}}_{k,m_2,n}$. 
\begin{lemma}[Decay of angular factor]\label{lem:angular-bound}
There exist constants $C,c > 0$ determined by $s$, so that
\begin{equation}\label{eqn:angular_decay}
\bigl|\Psi^{\mathrm{ang}}_{k,m_2,n}\bigr|
\leq C 
R^{-1} \exp\!\Bigl(-c \bigl|m_2/2 + n/R\bigr|^{1/s}\Bigr) 
\quad \textrm{for } k=1,\cdots,m(j) \textrm{ and } m_2,n \in \Z.
\end{equation}
\end{lemma}
\begin{proof}
Condition (A3) in Section \ref{sec:cutoffs} asserts that $\eta_{0,k}$ is in the Gevrey class $\mathcal{G}^s$ with constants $(C_1,C_2R)$. (Recall that $m(0) = 2^{j_{\max}} = R$.) Furthermore, condition (A2) asserts that $\eta_{0,k}$ is supported on a set in $S^1$ of measure at most $\widetilde{C} m(0)^{-1} = \widetilde{C} R^{-1}$. Thus, according to Lemma \ref{lemma:Gevrey:Fourier_decay_periodic}, the Fourier coefficients of $\eta_{0,k}$ satisfy the decay bound
\begin{equation}\label{eq:Gevrey--angular}
\abs{\widehat{\eta_{0,k}}(n)}\leq C R^{-1} \exp\Big(-c \abs{n/R}^{1/s}\Big).
\end{equation}
For the angular factor,
\[
\Psi^{\mathrm{ang}}_{k,m_2,n}
= 2 \pi \cdot \widehat{\eta_{0,k}}\!\bigl(- ( m_2 R/2 + n)\bigr).
\]
Using the decay estimate \eqref{eq:Gevrey--angular}, we obtain \eqref{eqn:angular_decay}.
\end{proof}

In Lemmas \ref{lem:radial-bounds1} and \ref{lem:radial-bounds2}, we state the decay estimates on the radial factor $\Psi^{\mathrm{rad}}_{j,m_1,n}$ for large $n$ and large $m_1$. For this, we need the following elementary estimates for the Bessel functions.

\begin{lemma}[Bessel function estimates]\label{b.Bessel} For all $n\in \Bbb Z$, $n \neq 0$, $|t| < |n|$, 
\begin{equation}\label{first-bound}
|J_n(t)| \leq 
 \left( \frac{e^2 |t|}{2|n|} \right)^{|n|}.
\end{equation}
For all $n\in \Bbb Z$, $k \geq 0$, and $t \in \R$,
\begin{equation}\label{second-bound}
|\partial^k_t J_n(t)|\le 1.
\end{equation}
\end{lemma}
\begin{proof} 
The series representation of the Bessel function $J_n(t)$ for integral $n > 0$ is
\[
J_n(t) = \sum_{k=0}^\infty \frac{(-1)^k}{k! (n+k)! } \left( \frac{t}{2} \right)^{n+2k}
\]
Using that $(n+k)! \geq n ! \cdot (n+1)^k $ and $n! \geq (n/e)^n$,
\[
\begin{aligned}
|J_n(t)| \leq \sum_{k=0}^\infty \frac{1}{k! (n+k)! } \left( \frac{|t|}{2} \right)^{n+2k} &\leq \left(\frac{e|t|}{2n}\right)^n \sum_{k=0}^\infty \frac{1}{k! (n+1)^k } \left( \frac{|t|}{2} \right)^{2k} \\
&= \left(\frac{e|t|}{2n}\right)^n e^{t^2/(4(n+1))} \mbox{ for } n > 0.
\end{aligned}
\]
Using that $J_{-n}(t) = (-1)^n J_n(t)$, we deduce that
\[
|J_n(t)| \leq e^{t^2/(4(|n|+1))} \left( \frac{e|t|}{2|n|} \right)^{|n|} \mbox{ for all }\
 n \in \Z, n \neq 0.
\]
In particular,
\[
|J_n(t)| \leq e^{|n|/4} \left( \frac{e |t|}{2|n|} \right)^{|n|} 
\leq \left( \frac{e^2 |t|}{2|n|} \right)^{|n|} \quad \mbox{ for } |t| < |n|,
\]
completing the proof of \eqref{first-bound}.

For $n\in\mathbb{Z}$, consider the integral representation of $J_n$,
\[
J_n(t)=\frac{1}{\pi}\int_{0}^{\pi}\cos\!\big(n\theta-t\sin\theta\big)\,d\theta,
\qquad t\in\mathbb{R}.
\]
For $k\in\mathbb{N}$, by the Chain rule we have $|\partial_t^k \cos\!\big(n\theta-t\sin\theta\big)| \leq |\sin \theta|^k$. Thus,
\[
|\partial_t^k J_n(t)| =\frac{1}{\pi} \left| \int_{0}^{\pi}\partial_t^k \cos\!\big(n\theta-t\sin\theta\big)\,d\theta \right| \leq \frac{1}{\pi} \int_0^\pi |\sin(\theta)|^k \,  d \theta \leq 1,
\]
completing the proof of \eqref{second-bound}. Note this proof actually yields the refinement $|\partial_t^k J_n(t)| \leq c_k < 1$ for $k \geq 1$. \end{proof}

\begin{lemma}[Decay of radial factor in $n$]\label{lem:radial-bounds1} 
For any $0\leq \rho\leq 1$, $j<0$, and $m_1, n \in \Z$, we have 
\[
\bigl|\Psi^{\mathrm{rad}}_{j,m_1,n}(\rho)\bigr|
\leq R 2^{j/2} \min\{ 1, (10R/|n|)^{|n|} \}.
\]
\end{lemma}
\begin{proof}
For the first bound (decay in $n$), observe that by \eqref{rad}
\[
|\Psi^{\mathrm{rad}}_{j,m_1,n}(\rho)| \leq \int_0^\infty 2^{-j/2} r \phi_j(r) |J_n(r \rho)| d r,
\]
where the $\phi_j$ is supported on the interval $I_j = [R - 1.1 \cdot 2^j, R - 0.9\cdot 2^{j-1}]$ of size at most $2^j$ and $0 \leq \phi_j \leq 1$. Hence, $r\leq R$ in the above integral. And also, $\rho \leq 1$. So, $r \rho \leq R$. Provided $|n| \geq 10 R$ we have $|n| > r \rho$, and so by \eqref{first-bound},
\[
|\Psi^{\mathrm{rad}}_{j,m_1,n}(\rho)| \leq R 2^{j/2} \left( \frac{e^2 r \rho}{2|n|} \right)^{|n|} \leq R 2^{j/2} \left( \frac{10R}{|n|} \right)^{|n|}.
\]
On the other hand, for $|n| \leq 10R$ we have the complementary bound, using that $|J_n(r \rho)| \leq 1$,
\[
|\Psi^{\mathrm{rad}}_{j,m_1,n}(\rho)| \leq R 2^{j/2}.
\]
Combining the two estimates for the different ranges of $n$, we obtain the bound
\[
|\Psi^{\mathrm{rad}}_{j,m_1,n}(\rho)| \leq R 2^{j/2} \min \{1, (10 R/|n|)^{|n|} \}.
\]
\end{proof}
\begin{lemma}[Decay of radial factor in $m_1$]\label{lem:radial-bounds2} 
There exist constants $c, C > 0$ determined by $s$ such that the following holds. For any $0\leq \rho\leq 1$, $j<0$, $m_1, n \in \Z$, 
\[
\bigl|\Psi^{\mathrm{rad}}_{j,m_1,n}(\rho)\bigr|
\leq C R 2^{j/2} \exp(- c |m_1|^{1/s}).
\]
\end{lemma}
\begin{proof}
For fixed $\rho$, we first estimate the Fourier transform of the function $q(r) = (r/R) \phi_j(r) J_n(r \rho)$. According to property (R2) in Section \ref{sec:cutoffs}, the function $q$ is supported on an interval $I_j = [R - 1.1 \cdot 2^j, R - 0.9 \cdot 2^{j-1}]$ of size at most $2^j$ contained in $[0,R]$. Property (R3) in Section \ref{sec:cutoffs} asserts that
\[
| \partial^k_r \phi_j(r)| \leq 2^{-j k} C_1 C_2^k (k!)^s.
\]
That is, the function $\phi_j(r)$ is Gevrey--$s$ with constants $(C_1,C_22^{-j})$ on the domain $[0,R]$, where $C_1,C_2$ are constants determined by $s$. Similarly, the function $r \mapsto r/R$ is Gevrey-$s$ with constants $(1,1)$ on the domain $[0,R]$. Finally, the function $r \mapsto J_n(r \rho)$ is Gevrey--$s$ with constants $(1,1)$ on $[0,R]$ (see \eqref{second-bound}). 

According to Lemma \ref{lemma:Gevrey:multiplication}, we deduce that $q(r)$ is Gevrey--$s$ with constants $(C_1,4\max\{C_2 2^{-j},1\})$. Hence, since $2^{-j} \geq 1$, we deduce that $q(r)$ is Gevrey--$s$ with constants $(A,B 2^{-j})$ for constants $A = C_1,B = \max\{4C_2,1\}$ determined by $s$. Furthermore, this function is supported on an interval of size at most $2^j$ contained in $[0,R]$. According to Lemma \ref{lemma:Gevrey:Fourier_decay}, 
\[
|\widehat{q}(\omega)| \leq C 2^{j} \exp(-c | 2^{j} \omega|^{1/s}),
\]
where $C,c > 0$ depend on $s$. By \eqref{rad}, $\Psi^{\mathrm{rad}}_{j,m_1,n}(\rho)$ is related to the Fourier transform of $q$ via
\[
\Psi^{\mathrm{rad}}_{j,m_1,n}(\rho) = (2 \pi)^{1/2} 2^{-j/2} R \cdot \widehat{q}(- m_1 2^{-j-1}),
\]
yielding that
\[
|\Psi^{\mathrm{rad}}_{j,m_1,n}(\rho)| \leq C R 2^{j/2} \exp(-c' | m_1 |^{1/s}),
\]
with $c' = c/2^{1/s}$, proving the desired bound.
\end{proof}

\begin{lemma}\label{lem:type2_decay}
For any $\mathbf{m} \in \Z^2$, $j < 0$, $k = 1,\cdots, 2^{j_{\max}}$, and any $\xi \in \R^2$, $|\xi| \leq 1$,
\[
| \widehat{\psi_{j,k,\mathbf{m}}}(\xi)| \leq C R 2^{j/2} \exp(-c |m_1|^{1/s}) \exp(-c |m_2|^{1/s}),
\]
where $C,c > 0$ are constants determined by $s$.
\end{lemma}

\begin{proof}
By the triangle inequality,
\[
|\widehat{\psi_{j,k,\mathbf{m}}}(\xi)| \leq \sum_{n \in \Z} | \Psi^{\mathrm{rad}}_{j,m_1,n}(\rho)| | \Psi^{\mathrm{ang}}_{k,m_2,n}|.
\]
By Lemma \ref{lem:angular-bound} and Lemma \ref{lem:radial-bounds1},
\[
\begin{aligned}
|\widehat{\psi_{j,k,\mathbf{m}}}(\xi)| &\lesssim_s \sum_{n \in \Z} R 2^{j/2} \min\{ 1, (10 R/|n|)^{|n|}\} R^{-1} \exp(-c | m_2/2 + n/R|^{1/s}) \\
& =  \sum_{n \in \Z} 2^{j/2} \min\{ 1, (10 R/|n|)^{|n|}\} \exp(-c | m_2/2 + n/R|^{1/s}).
\end{aligned}
\]
We split the sum on $n$ into two bins, \[
\begin{aligned}
&H_1 = \{ n : |m_2/2 + n/R| > |m_2|/4 \}, \\
&H_2 = \{ n : |m_2/2 + n/R | \leq |m_2|/4\}.
\end{aligned}
\]

Observe that
\begin{align*}
\sum_{n \in H_1} \exp\!\big(- c | m_2/2 + n/R|^{1/s}\big)
&= \sum_{p : |p/R| >|m_2|/4} \exp\!\big(-c |p/R|^{1/s}\big) \\
&\lesssim R \sum_{k \geq |m_2|/4} \exp\!\big(-c \cdot k ^{1/s}\big) \\
&\lesssim R \exp\!\big(-c' |m_2|^{1/s}\big).
\end{align*}
For the first identity, we make the change of variable $p = Rm_2/2 + n$ in the series. (Observe, $Rm_2/2 \in \Z$.) For the next inequality, we split the sum on $p$ into the sum of a finite sum over intervals of $R$ consecutive integers (on each interval the summand is essentially constant, and there are $R$ integers in each interval).

Therefore, we have
\[
\begin{aligned}
\sum_{n \in H_1} 2^{j/2} \min\{ 1, (10R/|n|)^{|n|}\} \exp(-c | m_2/2 + n/R|^{1/s}) &\leq 2^{j/2} \sum_{n \in H_1} \exp(-c | m_2/2 + n/R|^{1/s}) \\
&\lesssim R 2^{j/2} \exp(-c' |m_2|^{1/s}),
\end{aligned}
\]
where in the first inequality we upper bounded the $\min \{ \cdots \}$ by $1$.

By the triangle inequality,
\begin{equation}\label{eq:n-over-R-lower}
\left|\frac{n}{R}\right|
\ge \left| \frac{m_2}{2} \right|-\left|\frac{n}{R}+\frac{m_2}{2}\right| \geq \frac{|m_2|}{2} - \frac{|m_2|}{4} = \frac{|m_2|}{4} \qquad (n \in H_2).
\end{equation}
Consequently,
\begin{equation}\label{eq:10R-over-n}
\frac{10R}{|n|}\le \frac{C}{|m_2|} \qquad (n\in H_2).
\end{equation}
If $|m_2|\ge C$, then $C/|m_2|\le 1$, and combining \eqref{eq:n-over-R-lower}--\eqref{eq:10R-over-n} yields
\[
\left(\frac{10R}{|n|}\right)^{|n|}
\le \left(\frac{C}{|m_2|}\right)^{|n|}
\le \left(\frac{C}{|m_2|}\right)^{R|m_2|/4}
\le \left(\frac{C}{|m_2|}\right)^{|m_2|/4}.
\]
If instead $|m_2|\le C$, then trivially
\[
1 \le \left(\frac{C}{|m_2|}\right)^{|m_2|/4}.
\]
Thus, in all cases,
\begin{equation}\label{eq:min-bound}
\min\!\left\{1,\left(\frac{10R}{|n|}\right)^{|n|}\right\}
\le \left(\frac{C}{|m_2|}\right)^{|m_2|/4}
\qquad (n\in H_2).
\end{equation}

Moreover, $\#H_2\le C|m_2|R$. Since $\exp(\cdots)\le 1$, we deduce
\begin{align*}
\sum_{n\in H_2} 2^{j/2}\min\!\left\{1,\left(\frac{10R}{|n|}\right)^{|n|}\right\}
\exp\!\big(-c|m_2/2+n/R|^{1/s}\big)
&\le \sum_{n\in H_2} 2^{j/2}\min\!\left\{1,\left(\frac{10R}{|n|}\right)^{|n|}\right\} \\
&\le 2^{j/2} (\#H_2) \left(\frac{C}{|m_2|}\right)^{|m_2|/4} \\
&\le 2^{j/2} (C|m_2|R) \left(\frac{C}{|m_2|}\right)^{|m_2|/4} \\
&\lesssim R 2^{j/2}\left(\frac{C'}{|m_2|}\right)^{|m_2|/4},
\end{align*}
where in the last step we absorbed the polynomial factor $|m_2|$ into the exponential term $(C')^{|m_2|/4}$.

Combining the bounds on the sum over $H_1$ and the sum over $H_2$, we get
\[
\begin{aligned}
|\widehat{\psi_{j,k,\mathbf{m}}}(\xi)| &\lesssim_s \sum_{n \in H_1} \cdots + \sum_{n \in H_2} \cdots \\\notag
&\lesssim R 2^{j/2} \left[\exp(-c |m_2|^{1/s}) + (C/|m_2|)^{|m_2|/4} \right] \\
&\lesssim R 2^{j/2} \exp(-c|m_2|^{1/s}).
\end{aligned}
\]

We then prove a separate inequality, using Lemma \ref{lem:radial-bounds2}, $|\Psi^{\mathrm{rad}}_{j,m_1,n}(\rho)|\lesssim_s R 2^{j/2} \exp(-c |m_1|^{1/s})$ and the estimate on the angular factor in \eqref{eqn:angular_decay}, 
\[
|\widehat{\psi_{j,k,\mathbf{m}}}(\xi)| \lesssim_s R 2^{j/2} \exp(-c |m_1|^{1/s}) \sum_{n \in \Z} R^{-1} \exp(-| m_2/2 + n/R|^{1/s}) \lesssim R 2^{j/2} \exp(-c|m_1|^{1/s}).
\]
Taking the geometric mean of the previous two inequalities, we get
\[
|\widehat{\psi_{j,k,\mathbf{m}}}(\xi)| \lesssim_s R 2^{j/2} \exp(-c|m_2|^{1/s}) \exp(-c|m_1|^{1/s}), 
\]
 with a different positive constant $c$. 
\end{proof}
We note the following consequence of Lemma \ref{lem:type2_decay}. Given $\delta \in (0,1/2)$, define
\begin{equation}\label{defn:partition_bdry}
\left\{
\begin{aligned}
&\cI_{1}^{\mathrm{bdry}} = \{(j,k,\mathbf{m}) : j \leq - \log(1/\delta) \mbox{ or } |m_1| > C \log(1/\delta)^s \mbox{ or } |m_2| > C \log(1/\delta)^s \}, \\
&\cI_3^{\mathrm{bdry}} = \{ (j,k,\mathbf{m}) : - \log(1/\delta) < j < 0 \mbox{ and } |m_1| \leq C \log(1/\delta)^{s} \mbox{ and } |m_2| \leq C \log(1/\delta)^{s} \}
\end{aligned}
\right.
\end{equation}
for some constant $C = C_s > 1$, which is picked large enough to ensure 
$$\sum_{|m_1| > C \log(1/\delta)^s} \exp(-c|m_1|^{1/s})^2 \leq \delta.$$
(Here, $c$ is the constant in Lemma \ref{lem:type2_decay}.) Applying 
 Lemma \ref{lem:type2_decay} and summing over $(j,k, \mathbf{m})\in \cI_1^{\mathrm{bdry}}$ yields 
\begin{align}
\sum_{(j,k,\mathbf{m}) \in \cI_1^{\mathrm{bdry}}} & \| \widehat{\psi_{j,k,\mathbf{m}}} \|_{L^2(S)}^2 \lesssim_s \sum_{j \leq - \log(1/\delta)} \sum_{k=1}^{2^{j_{\max}}} \sum_{m_1,m_2=-\infty}^\infty R^2 2^j \exp(-c |m_1|^{1/s})^2 \exp(-c |m_2|^{1/s})^2 \notag \\
& + 2 \sum_{j <0} \sum_{k=1}^{2^{j_{\max}}} \sum_{|m_1| > C \log(1/\delta)^s} \sum_{m_2=-\infty}^\infty R^2 2^j \exp(-c |m_1|^{1/s})^2 \exp(-c |m_2|^{1/s})^2 \notag\\
&\leq C R^3 \delta. \label{eqn:energy_type2}
\end{align}
Here, we have used that $2^{j_{\max}} = R$ and that the series $\sum_{m = - \infty}^\infty \exp(-c |m|^{1/s})^2$ is convergent. By definition of $\cI_3^{\mathrm{bdry}}$, the index 
 $j$ ranges over at most $O(\log(1/\delta))$ integers, while $m_1$ and $m_2$ range over sets of size 
$O(\log(1/\delta)^{s})$ each. For each $(j, m_1, m_2)$, the angular index $k$ takes $O(R)$ values. Hence, 
\begin{equation}\label{eqn:res_bound2}
\#( \cI_3^{\mathrm{bdry}}) \leq C R \log(1/\delta)^{1 + 2s}.
\end{equation}
Observe that 
\[
\cI^{\mathrm{bdry}} = \{ (j,k,\mathbf{m}) : j < 0, 1 \leq k \leq 2^{j_{\max}}, \mathbf{m} \in \Z^2\} = \cI_1^{\mathrm{bdry}} \cup \cI_3^{\mathrm{bdry}}.
\]
Therefore, the boundary wave packets can be categorized efficiently.

\subsection{Proof of Proposition~\ref{prop1}}

Let $ \varepsilon \in (0,1/2)$ and $R = 2^{j_{\max}} \geq 2$, $j_{\max}\in \Bbb N$, be given. 
Fix $\delta \in (0,1/2)$ to be determined below. Define the partition sets in \eqref{defn:partition_int} and \eqref{defn:partition_bdry} for this $\delta$, where $A = A_s$ in \eqref{defn:partition_int} is picked to be a large enough constant determined by $s$ as per Lemma \ref{lem:energy_type1}. Let $\cI_1 = \cI_1^{\mathrm{int}} \cup \cI_1^{\mathrm{bdry}}$, $\cI_2 = \cI_2^{\mathrm{int}}$ and $\cI_3 = \cI_3^{\mathrm{int}} \cup \cI_3^{\mathrm{bdry}}$, so that $\cI_1 \cup \cI_2 \cup \cI_3 = \cI$. 
According to \eqref{eqn:res_bound1} and \eqref{eqn:res_bound2}, we have 
 \[
 \begin{aligned}
 \#(\cI_3)
 = \#(\cI^{\mathrm{int}}_3)+\#(\cI^{\mathrm{bdry}}_3)
 &\leq
 C_{s,S} R \max\{  \log R \log(1/\delta)^{s}, \log(1/\delta)^{2s} \} + C R\log(1/\delta)^{1+2s} \\
 &\leq C_{s,S}' R \max\{  \log R \log(1/\delta)^{s} , \log(1/\delta)^{1+2s} \}.
 \end{aligned}
 \]
According to Lemma \ref{lem:energy_type1} and \eqref{eqn:energy_type2}, for some $C>1$ determined by $s$,
\[
\sum_{\cI_1} \| \widehat{\psi_\nu} \|_{L^2(S)}^2 + \sum_{\cI_2} \| \widehat{\psi_\nu} \|_{L^2(\R^2 \setminus S)}^2 \leq C R^3 \delta.
\]
We now set $\delta = \varepsilon^2/(C R^3)$ so that 
\[
\sum_{\cI_1} \| \widehat{\psi_\nu} \|_{L^2(S)}^2 + \sum_{\cI_2} \| \widehat{\psi_\nu} \|_{L^2(\R^2 \setminus S)}^2 \leq \varepsilon^2.
\]
Then $\log(1/\delta) \leq C\log(R/ \varepsilon)$ and $\log R \leq \log(R/\varepsilon)$. Thus, $\#(\cI_3) \lesssim_{s,S} R \log(R/ \varepsilon)^{1+2s}$. This completes the proof of Proposition \ref{prop1}.

\section{Proof of Theorem \ref{thm:newmain} and Theorem \ref{thm2}} 
\label{sec:mainresult}

Theorem \ref{thm2} is immediate from Proposition \ref{prop2} and Proposition \ref{prop1}.

To prove Theorem \ref{thm:newmain}, we use the following lemma from \cite{marceca2023} which is an extension of a lemma from \cite{israel15eigenvalue}:

\begin{lemma}\label{lem:Romero_lem}
Let $T : \cH \rightarrow \cH$ be a positive, compact, self-adjoint operator on a Hilbert space $\cH$ with $\| T \| \leq 1$ and eigendecomposition $T = \sum_{n \geq 0} \lambda_n \langle \cdot, f_n \rangle f_n$. 
Let $\{\phi_i \}_{i \in \cI}$ be a unit--norm frame for $\cH$ with lower frame constant $A$. If $\cI = \cI_1 \cup \cI_2 \cup \cI_3$, and
\[
\sum_{i \in \cI_1} \| T \phi_i \|^2 + \sum_{i \in \cI_2} \| (I - T) \phi_i \|^2 \leq \frac{A}{2} \varepsilon^2
\]
then $\# \{ \lambda_n(T) \in (\varepsilon,1- \varepsilon) \} \leq \frac{2}{A} \# \cI_3$.
\end{lemma}

We consider the unit--norm frame $\{ \psi_{\nu}\}_{\nu \in \cI}$ for $L^2(D(R))$ with frame constants $0 < A < B < \infty$, satisfying the conditions of Theorem \ref{thm2}.

Let $\varepsilon_0:=\sqrt{A/2} \varepsilon$. 
Using the partition $\cI = \cI_1 \cup \cI_2 \cup \cI_3$ from Theorem \ref{thm2} 
(with $\varepsilon_0$ in place of $\varepsilon$), and using that $P_{D(R)}$ is an orthogonal
projection (hence a contraction) and that $P_{D(R)}\psi_\nu=\psi_\nu$, we have
\[
\begin{aligned}
\sum_{\nu\in I_1}\|T_R\psi_\nu\|_2^2 + \sum_{\nu\in I_2}\|(I-T_R)\psi_\nu\|_2^2
&\le \sum_{\nu\in I_1}\|B_{S}\psi_\nu\|_2^2
 + \sum_{\nu\in I_2}\|(I-B_{S})\psi_\nu\|_2^2 \\
&= \sum_{\nu\in I_1}\|\widehat{\psi_\nu}\|_{L^2(S)}^2
 + \sum_{\nu\in I_2}\|\widehat{\psi_\nu}\|_{L^2(\mathbb{R}^2\setminus S)}^2 \\
&\le \varepsilon_0^2
= (A/2)\varepsilon^2.
\end{aligned}
\]
Therefore, by Lemma \ref{lem:Romero_lem}, we deduce that
\[
\# \{ n : \lambda_n(T_R) \in ( \varepsilon,1- \varepsilon)\} \leq \frac{2}{A} \#(\cI_3) \lesssim_{s,S} R \log(R/ \varepsilon)^{1+2s}.
\]

% later in the document:
\begin{appendices}

\section{Properties of Gevrey functions}\label{appendix:Gevrey}

We require the following technical lemma.

\begin{lemma}
\label{lem:tech}
Given $s \geq 1$,
\[
\inf_{\ell \in \Z_{\geq 0}} (\ell !)^s X^{- \ell} \leq C \exp(-c X^{1/s}) \quad \forall X > 0,
\]
where $C=\exp(1/2)$ and $c=\frac{1}{2e}$.
\end{lemma}
\begin{proof}
Suppose first $X \geq e^s $. Set $\ell = \biggl\lfloor \frac{X^{1/s}}{e} \biggr\rfloor \geq 1$, and observe $\ell^s \leq \frac{X}{e^s}$. Hence,
\[
\ln(\ell^s) - \ln(X) \leq \ln\left(\frac{X}{e^s}\right) - \ln(X) = -\ln(e^s) \leq - 1.
\]
Observe that $\ell!\leq \ell^\ell$, so
\[
(\ell!)^s X^{-\ell} \leq \ell^{\ell s} X^{-\ell} = \exp( \ell (\ln(\ell^s) - \ln(X) )) \leq \exp(- \ell ) \leq \exp( - c X^{1/s}),
\]
where $c=\frac{1}{2e}$. For the last inequality, note that $\ell=\biggl\lfloor \frac{X^{1/s}}{e } \biggr\rfloor \geq \frac{X^{1/s}}{2e}$. (Indeed, $\lfloor y\rfloor\geq y/2$ for $y \geq 1$.)

Now suppose $X < e^s $. If $\ell = 0$ then $(\ell!)^s X^{-\ell} = 1$. So it suffices to show that 
\[
1 \leq C \exp(-c X^{1/s}) \quad \forall X\in (0,e^s).
\]
That is,
$C\geq \exp(cX^{1/s})$ for $X \in (0, e^s )$. Taking $C= \exp( ec ) = \exp(1/2)$ gives the desired inequality, 
completing the proof of the lemma.
\end{proof}

\begin{lemma}[Fourier decay of Gevrey functions]\label{lemma:Gevrey:Fourier_decay}
Let \(s \geq 1\). 
Fix $\phi \in C^\infty(\R^d)$ to be Gevrey--$s$ with constants $(A,B)$ and with support contained in a bounded region of Lebesgue measure at most $H$. Then its Fourier transform satisfies the estimate
\begin{equation}\label{bd:Gevrey:FT}
|\widehat{\phi}(\xi)| \leq C A H \exp(- c (| \xi|/B)^{1/s}) \qquad (\xi \in \R^d)
\end{equation}
where $C = \exp(d/2)$ and $c=\frac{1}{2e}$.
\end{lemma}
\begin{proof}
We estimate the Fourier transform of \(\phi\),
\[
\widehat{\phi}(\xi) = (2 \pi)^{-d/2} \int_{\R^d} \phi(x) e^{ - i \langle \xi , x \rangle} \, d x.
\]
Fix a multiindex $\alpha=(\alpha_1,\alpha_2, \dots, \alpha_d) \in \Z_{\geq 0}^d$ to be determined below. For $\xi \neq 0$,
$$ e^{-i\langle \xi, x \rangle} = (-i\xi)^{-\alpha} \partial^\alpha_x e^{-i\langle \xi, x \rangle},
$$
where $(-i \xi)^{-\alpha} = (-i \xi_1)^{-\alpha_1}\cdots (-i\xi_d)^{-\alpha_d}$. Thus,
\[
\widehat{\phi}(\xi)
= (2\pi)^{-d/2} ( -i\xi)^{-\alpha} \int_{\mathbb{R}^d} \phi(x)  \partial^\alpha_x e^{- i\langle \xi, x \rangle} \, dx.
\]
Now, we integrate by parts to obtain
\[
\widehat{\phi}(\xi) = (2 \pi)^{-d/2}
 (-1)^{|\alpha|} ( - i\xi)^{-\alpha} \int_{\R^d} \partial^\alpha_x \phi(x) e^{-i \langle \xi , x \rangle} \, d x.
\]
Note that $|(-i \xi)^{-\alpha}| = |\xi_1|^{-\alpha_1} \cdots |\xi_d|^{-\alpha_d}$. Therefore, given that $\phi$ is Gevrey--$s$ with constants $(A,B)$, and is supported on a set of measure $\leq H$,
\[
|\widehat{\phi}(\xi)| \leq H A B^{|\alpha|} (\alpha !)^s \prod_{j=1}^d |\xi_j|^{-\alpha_j} = A H \prod_{j=1}^d \big( B^{\alpha_j} (\alpha_j !)^s |\xi_j|^{-\alpha_j} \big). \]
We now choose $\alpha \in \Z_{\geq 0}^d$ to minimize the term on the right-hand side using Lemma \ref{lem:tech}. (We apply this lemma with $X = |\xi_j|/B$ for each fixed $j$.) We determine that, for $C=\exp(d/2)$ and $c=\frac{1}{2e}$,
\[
|\widehat{\phi}(\xi)| \leq C A H \prod_{j=1}^d \exp( - c (|\xi_j|/B)^{1/s}) \leq C A H \exp( - c(|\xi|/B)^{1/s}).
\]
Here, the last inequality holds since $s\geq 1$ and the function $t
\to t^{1/s}$ is subadditive, together with the fact that the $\ell^1_d$ norm of $\xi$ is greater than or equal to the Euclidean norm of $\xi$. 

The preceding bound holds provided $\xi \neq 0$. If $\xi = 0$ then
\[
|\widehat{\phi}(\xi)| \leq \| \phi \|_{L^1} \leq \| \phi \|_{L^\infty} | \mathrm{supp}(\phi)| \leq A H.
\]
\end{proof}

There is an analogue of the previous lemma for periodic Gevrey functions on the circle. The proof is almost identical to that of the previous lemma and we leave its verification to the reader.

\begin{lemma}[Fourier decay of periodic Gevrey functions]\label{lemma:Gevrey:Fourier_decay_periodic}
 Suppose $\eta$ is a $C^\infty$ function defined on the unit circle $S^1$. Suppose $\eta$ obeys the Gevrey type bounds, $\| \partial_\theta^k \eta \|_{L^\infty} \leq A B^k (k!)^s$ for all $k \geq 0$, for some $s \geq 1$, and suppose $\eta$ is supported on a region in $S^1$ of measure at most $H$. (Here, necessarily $H \leq 2 \pi$.) Then the Fourier coefficients of $\eta$ satisfy the estimates
 \[
 |\widehat{\eta}(n)| \leq C A H \exp(- c (|n|/B)^{1/s}) \qquad (n \in \Z),
 \]
 where $C,c > 0$ are absolute constants (independent of $A,B,s$).
\end{lemma}

\begin{lemma}\label{lemma:Gevrey:multiplication}
Let \(s_1,s_2 \geq 1\). 
For \(i=1,2\), let \(\phi_i\) be a Gevrey-\(s_i\) functions on $\R$ with constants \((A_i,B_i)\). Then \(\phi_1(x) \cdot \phi_2(x)\) is Gevrey--\(\max\{s_1,s_2\}\) with constants \((A_1A_2,2\max\{B_1,B_2\})\).
\end{lemma}

\begin{proof}
Let \(\phi(x) = \phi_1(x)\phi_2(x)\). The proof is an application of the Leibniz rule
\begin{equation}\label{eq:Leibniz}
\partial^k{\phi} = \sum_{j=0}^{k} \binom{k}{\j} \partial^j{\phi_1} \cdot \partial^{k-j}\phi_2
.\end{equation}
By \eqref{eq:Leibniz} and the triangle inequality, 
\begin{align*}
\|\partial^k{\phi}\|_{\infty}
&\leq 
\sum_{j=0}^k \binom{k}{j} \|\partial^j{\phi_1} \cdot \partial^{k-j}{\phi_2}\|_{\infty}
\\&\leq 
\sum_{j=0}^k \binom{k}{j} A_1(B_1)^j (j!)^{s_1} A_2(B_2)^{k-j} ([k-j]!)^{s_2}.
\end{align*}
Let \(B=\max\{B_1,B_2\}\) and \(s=\max\{s_1,s_2\}\). Since \(j! \cdot [k-j]! \leq k!\) for all \(0 \leq j \leq k\), we continue the estimation to find that 
\begin{align*}
\|\partial^k{\phi}\|_{\infty}
&\leq 
A_1A_2 \sum_{j=0}^k \binom{k}{j} B^j (j!)^{s} B^{k-j} ([k-j]!)^{s} 
\\&\leq 
A_1A_2 2^k \cdot (k!)^s B^k
.\end{align*}
This completes the proof. 
\end{proof}

\section{Construction of cutoffs}\label{cutoffs}

\begin{lemma}[Construction of radial cutoffs]\label{lem:radial-cutoffs}
Fix $s>1$ and $R=2^{j_{\max}}$,$j_{\max}\in \Bbb N$. There exist functions
$\{\phi_j\}_{j\le j_{\max}}\subset C^\infty([0,R])$ and absolute constants
$C_1,C_2>0$ (depending only on $s$) such that
\begin{enumerate}
\item[(R1)] $\sum_{j\le j_{\max}} \phi_j(r)^2=1$ for all $r\in[0,R)$,
\item[(R2)] $\supp(\phi_j)\subset I_j:=\big[R-1.1 2^j, R-0.9 2^{j-1}\big]\cap[0,R]$,
\item[(R3)] $\|\partial_r^k\phi_j\|_{L^\infty([0,R])}\le C_1 C_2^k (k!)^s 2^{-jk}$ for all $k\ge 0$,
\end{enumerate}
and the family $\{\phi_j\}$ is locally finite on $[0,R)$.
\end{lemma}

\begin{proof}
Fix $u\in C^\infty(\mathbb R)$ a Gevrey--$s$ ``mother" function such that
\[
0\le u\le 1,\qquad
\supp(u)\subset[0.45,1.1],\qquad
u\equiv 1 \text{ on } [1/2,1],
\]
and for some constants $A,B>0$ (depending only on $s$),
\begin{equation}\label{eq:template-gevrey}
\|\partial_r^ku\|_{L^\infty(\mathbb R)}\le A B^k (k!)^s,\qquad k\ge 0.
\end{equation}
(The existence of such $u$ is standard for $s>1$. For example, see \cite{israelmayeli2023acha}.) 
Fix $j\le j_{\max}$, and define the cutoffs on $[0,R)$ by
\begin{equation}\label{eq:pre-cutoff}
\widetilde\phi_j(r):=u\!\left(\frac{R-r}{2^j}\right).
\end{equation}
It is clear that $\supp(\widetilde\phi_j)\subset I_j$, where $I_j= [R-1.1 2^j, R-0.9 2^{j-1}],
$. 
Moreover, by the chain rule and
\eqref{eq:template-gevrey}, we obtain 
\begin{equation}\label{eq:pre-derivs}
\|\partial_r^k\widetilde\phi_j\|_{L^\infty([0,R])}
\le A B^k (k!)^s 2^{-jk},\qquad k\ge 0.
\end{equation}

Now we construct $\phi_j$: Define \begin{equation}\label{eq:W-def}
W(r):=\sum_{j\le j_{\max}}\widetilde\phi_j(r)^2,\qquad r\in[0,R). 
\end{equation}
There are absolute constants $0<c\leq C<\infty$ such that $c\leq W(r)\leq C$. 
We define $\phi_j$ as the normalized cutoffs
\begin{equation}\label{eq:phi-def}
\phi_j(r):=\frac{\widetilde\phi_j(r)}{\sqrt{W(r)}},\qquad r\in[0,R).
\end{equation}
By the support condition for the function $u$, it is clear that the sum \eqref{eq:W-def} is locally finite, since for any $r\in [0,R)$, only for two indices $j$ for which $\widetilde\phi_j(r)\neq 0$. 

By the definition of $\phi_j$, the properties (R1)--(R2) hold immediately. 
Next, we prove the property (R3) for $\phi_j$: On $\supp(\widetilde\phi_j)$, the sum $W(r)$ involves only $\widetilde\phi_{j-1}, \widetilde\phi_j, \widetilde\phi_{j+1}$. Therefore, by \eqref{eq:pre-derivs}, one obtains on $\supp(\widetilde\phi_j)$
the Gevrey--$s$ bounds
\[
\|\partial_r^k W\|_{L^\infty(\supp(\widetilde\phi_j))}
\le A' (B')^k (k!)^s 2^{-jk},\qquad k\ge 0,
\]
with constants $A',B'$ depending only on $s$ and $u$. Since $W(r)\in [c,C]$, and $c>0$, and the function $F(x)=x^{-1/2}$ is real analytic on $[c,C]$, by the stability of Gevrey functions under the composition with an analytic function, we obtain 
on 
$\supp(\widetilde\phi_j)$
\[
\|\partial_r^k (W^{-1/2})\|_{L^\infty(\supp(\widetilde\phi_j))}
\le A'' (B'')^k (k!)^s 2^{-jk},\qquad k\ge 0.
\]
Applying Lemma~\ref{lemma:Gevrey:multiplication} to $\phi_j = \widetilde\phi_j \cdot W^{-1/2}$ proves (R3). 
\end{proof}

\begin{lemma}[Construction of angular cutoffs]\label{lem:angular-cutoffs}
Fix $s>1$. For each $0\le j\le j_{\max}$ let $m(j):=2^{j_{\max}-j}$ and
\[
\Theta_{j,k}:=\Big[\tfrac{2\pi(k-1)}{m(j)}, \tfrac{2\pi k}{m(j)}\Big),
\qquad k=1,\dots,m(j).
\]
Set
\[
\Delta_j:=\frac{2\pi}{m(j)},\qquad
\Theta_{j,k}^\ast:=\Big[\tfrac{2\pi(k-1)}{m(j)}-0.05 \Delta_j, 
\tfrac{2\pi k}{m(j)}+0.05 \Delta_j\Big)
\subset \mathbb R,
\]
and view $\Theta_{j,k}^\ast$ periodically modulo $2\pi$.
Then there exist functions $\{\eta_{j,k}\}_{k=1}^{m(j)}\subset C^\infty(S^1)$, $S^1=[0, 2\pi)$, and
absolute constants $C_1,C_2>0$ (depending only on $s$) such that
\begin{enumerate}
\item[(A1)] $\sum_{k=1}^{m(j)}\eta_{j,k}(\theta)^2=1$ for all $\theta\in S^1$,
\item[(A2)] $\supp(\eta_{j,k})\subset \Theta_{j,k}^\ast$,
\item[(A3)] $\|\partial_\theta^k\eta_{j,k}\|_{L^\infty(S^1)}\le C_1 C_2^k (k!)^s m(j)^k$ for all $k\ge 0$.
\end{enumerate}
\end{lemma}

\begin{proof} 
Let $u\in C^\infty(\mathbb R)$ be a fixed Gevrey--$s$ template satisfying
\[
0\le u\le 1,\qquad \supp(u)\subset[-0.55, 0.55],\qquad
u\equiv 1 \text{ on }[-0.50, 0.50],
\]
and for some constants $A,B>0$ (depending only on $s$),
\begin{equation}\label{eq:ang-template-gevrey}
\|u^{(k)}\|_{L^\infty(\mathbb R)}\le A B^k (k!)^s,\qquad k\ge 0.
\end{equation}
For fixed $j$, define the angular windows on $\mathbb R$ by
\[
\widetilde\eta_{j,k}(\theta):=
u\!\left(\frac{\theta-\theta_{j,k}}{\Delta_j}\right),
\qquad
\theta_{j,k}:=\Big(k-\tfrac12\Big)\Delta_j,
\qquad k=1,\dots,m(j).
\]
Since $\supp(u)\subset[-0.55, 0.55]$, we have \[
\supp(\widetilde\eta_{j,k})
\subset \big[\theta_{j,k}-0.55 \Delta_j, \theta_{j,k}+0.55 \Delta_j\big]
=
\Big[\tfrac{2\pi(k-1)}{m(j)}-0.05 \Delta_j, \tfrac{2\pi k}{m(j)}+0.05 \Delta_j\Big]
=\Theta_{j,k}^\ast,
\]
which gives (A2) for $\widetilde\eta_{j,k}$.
Next set
\[
W_j(\theta):=\sum_{k=1}^{m(j)} \widetilde\eta_{j,k}(\theta)^2,
\qquad
\eta_{j,k}(\theta):=\frac{\widetilde\eta_{j,k}(\theta)}{\sqrt{W_j(\theta)}}.
\] 
By an argument entirely analogous to the radial cutoffs, one can prove that the angular cutoffs $\{\eta_{j,k}\}$ satisfy \emph{(A1)}--\emph{(A3)} with constants uniform in $j$ and $k$. Since the proof is routine, we omit the details.
\end{proof}
 
%\needspace{12\baselineskip}
\FloatBarrier
\section{List of Notation}\label{sec:notations} 
\begin{table}[H]%[ht]
\centering
\setlength{\tabcolsep}{4pt}
\small
\renewcommand{\arraystretch}{0.9}
\begin{threeparttable}
\caption{Notation used throughout.}
\label{tab:notation}
\small
\begin{tabularx}{\textwidth}{@{}>{\raggedright\arraybackslash}p{0.38\textwidth} X@{}}
\toprule
\textbf{Symbol} & \textbf{Meaning} \\
\midrule
\multicolumn{2}{@{}l@{}}{\textit{General conventions}}\\
\addlinespace[2pt]
$A\lesssim B$ & $A\le C B$ with constant $C>0$ independent of scale parameter(s). \\
$A\approx B$ & Both $A\lesssim B$ and $B\lesssim A$. \\
$|E|$ & Lebesgue measure of $E\subset\mathbb{R}^2$. \\
$\mathbf{1}_E$ & Indicator function of a set $E$. \\
$\langle f,g\rangle$ & $L^2$ inner product; $\|f\|_2^2=\langle f,f\rangle$. \\
%\addlinespace[4pt]
\midrule
\multicolumn{2}{@{}l@{}}{\textit{Spaces and variables}}\\
\addlinespace[2pt]
$x=(x_1,x_2)$, $\xi=(\xi_1,\xi_2)$ & Spatial and frequency variables. \\
$(r, \theta) =(|x|,\arg(x))$, $(\rho, \phi) =(|\xi|,\arg(\xi))$& Polar coordinates. \\
${\bf m}$ & $2$-dimensional integer lattice point.\\
$s$ & Gevrey parameter. \\
%\addlinespace[4pt]
\midrule
\multicolumn{2}{@{}l@{}}{\textit{Fourier analysis}}\\
\addlinespace[2pt]
$\mathcal{F}$, $\mathcal{F}^{-1}$ & Fourier transform and inverse (with the chosen normalization). \\
$\widehat f$ & Fourier transform of $f$. \\
%\addlinespace[4pt]
\midrule
\multicolumn{2}{@{}l@{}}{\textit{Geometry and dilations}}\\
\addlinespace[2pt]
$S$ & Fixed frequency domain in $\mathbb{R}^2$. \\
$\cH^1(\partial S)$ & $1$D Hausdorff measure of boundary of $S$. \\
$F$ & Fixed spatial domain in $\mathbb{R}^2$.\\
$F(R)$ & Dilated spatial domain $F(R)= \{Rx:x\in F\}$ at scale $R$. \\
%\addlinespace[4pt]
\midrule
\multicolumn{2}{@{}l@{}}{\textit{Spatio--spectral limiting}}\\
\addlinespace[2pt]
$P_{F}$ & Spatial projection: $(P_{F} f)(x)=\mathbf{1}_{F}(x)f(x)$. \\
$B_{S}$ & Band-limiting projection to $S$: $B_{S} f=\mathcal{F}^{-1}(\mathbf{1}_{S}\widehat f)$. \\
$T_R$ & Limiting operator: $T_R=P_{F(R)}B_{S}P_{F(R)}$. \\
$\lambda_k(T_R)$ & Eigenvalues of $T_R$ (positive and nonincreasing). \\
%\addlinespace[4pt]
\midrule
\multicolumn{2}{@{}l@{}}{\textit{Wave packets, indices, and special functions}}\\
$\cI$ & Wave packet index set. \\
$\cI^{\mathrm{int}}$, $\cI^{\mathrm{bdry}}$ & The interior/boundary split of index set $\cI$.\\
$\nu$ & Packet index $\nu=(j,k,\mathbf{m})$. \\
$\psi_{j,k,\mathbf{m}}$ & Wave packet localized to sector $S_{j,k}$ with modulation index $\mathbf{m}$. \\
$J_n$ & Bessel function of the first kind (order $n$). \\
%\addlinespace[4pt]
\midrule
\multicolumn{2}{@{}l@{}}{\textit{Disc and discretization objects}}\\
 $j_{max}$ & 
Related to the dyadic assumption $R=2^{j_{max}}$.
 \\
 $m(j)$ &
$\displaystyle
m(j)=
\begin{cases}
2^{j_{max}-j}, & 0\le j\le j_{max},\\
2^{j_{max}}, & j<0.
\end{cases}$ \\
$\Theta_{j,k}$ & Angular arcs defining the partition. \\
$S_{j,k}$ & The radial--angular sectors. \\
$S_{j,k}^\ast$ & Enlargement of $S_{j,k}$. \\
%\addlinespace[4pt]
\midrule
\multicolumn{2}{@{}l@{}}{\textit{Cutoff windows}}\\
$\phi_j$
 & Radial cutoffs. \\
 $I_j$ & Support interval of $\phi_j$. \\
 $\eta_{j,k}$ & Angular cutoffs (paired with $\phi_j$ 
 in wave packet definition).\\
 \bottomrule
\end{tabularx}
\end{threeparttable}
\end{table}

\end{appendices}

\newpage
 
% \pagestyle{empty} 
% \addcontentsline{toc}{section}{References} 
 
\printbibliography

\end{document}